\documentclass[a4paper,reqno,12pt]{amsart}

\usepackage{hyperref}
\usepackage{amsmath}
\usepackage{amsthm}
\usepackage[T1]{fontenc}
\usepackage{amsmath}
\usepackage{graphicx}
\usepackage{epsf, subfigure, verbatim}
\usepackage{epsfig}
\usepackage{amssymb}
\usepackage{amsfonts}
\usepackage{empheq}
\usepackage{float}

\usepackage[margin=1.5in]{geometry}
\usepackage{srcltx}

\theoremstyle{plain} %default option

\newtheorem{theorem}[equation]{Theorem}
\newtheorem{proposition}[equation]{Proposition}
\newtheorem{lemma}[equation]{Lemma}
\newtheorem{conjecture}[equation]{Conjecture}
\newtheorem{main_theorem}{Theorem}
\numberwithin{equation}{subsection}
\newtheorem{question}[equation]{Question}

\newcommand{\lmoh}{\lambda^{-\frac{1}{2}}}
\newcommand{\loh}{\lambda^{\frac{1}{2}}}

\newcommand{\twopartdef}[4]
{
	\left\{
		\begin{array}{ll}
			#1 & \mbox{if } #2 \\
			#3 & \mbox{if } #4
		\end{array}
	\right.
}

\begin{document}

\title[Nodal sets and growth exponents of Laplace eigenfunctions]{Nodal sets and growth exponents of Laplace eigenfunctions on surfaces}
\author{Guillaume Roy-Fortin}
\thanks{The author has been supported by NSERC}

\begin{abstract}
We prove a result, announced by F. Nazarov, L. Polterovich and M. Sodin, that exhibits a relation between the average local growth of a Laplace eigenfunction on a closed surface and the global size of its nodal set. More precisely, we provide a lower and an upper bound to the Hausdorff measure of the nodal set in terms of the expected value of the growth exponent of an eigenfunction on disks of wavelength like radius. Combined with Yau's conjecture, the result implies that the average local growth of an eigenfunction on such disks is bounded by constants in the semi-classical limit. We also obtain results that link the size of the nodal set to the growth of solutions of planar Schr\"{o}dinger equations with small potential.
\end{abstract}

\maketitle

%-------------------------------------------------------------------------------------------------------
% 
%		SECTION 1 : INTRODUCTION AND MAIN RESULTS
%
%--------------------------------------------------------------------------------------------------------

\section{Introduction and main results}\label{s1}

\subsection{Nodal sets of Laplace eigenfunctions}
Let $(M,g)$ be a smooth, closed two-dimensional Riemannian manifold endowed with a $C^\infty$ metric $g$. Let $\{\phi_\lambda\}$, $\lambda \nearrow \infty$, be any sequence of eigenfunctions of the negative definite Laplace-Beltrami operator $\Delta_g$: 
\begin{equation}\label{main_eigenvalue_eq}
\Delta_g \phi_\lambda + \lambda \phi_\lambda = 0.
\end{equation}
In local coordinates, we write the Laplace-Beltrami operator as 
\begin{equation*}
\Delta_g = \frac{1}{\sqrt{g}} \sum_{i,j=1}^2 \frac{\partial}{\partial x_i} \left(g^{ij}\sqrt{g}\frac{\partial}{\partial x_j} \right).
\end{equation*}
The nodal set of $\phi_\lambda$ is the set $$Z_\lambda := \left\{ p \in M : \phi_\lambda(p) = 0 \right\}.$$ It is known, see \cite{Ch}, that $Z_\lambda$ is a smooth curve away from its finite singular set $$S_\lambda := \left\{ p \in M : \phi_\lambda(p) = \nabla \phi_\lambda(p) = 0 \right\}.$$ Nodal sets of Laplace eigenfunctions have been of interest since the discovery of the Chladni patterns and their asymptotic properties as $\lambda \nearrow \infty$ have been intensively studied, notably in the context of quantum mechanics. In that setting, the square of a normalized eigenfunction $\phi_\lambda$ represents the probability density of a free particle in the pure state corresponding to $\phi_\lambda$ and $Z_\lambda$ can be thought of as the set where such a particle is least likely to be found. Estimating the one dimensional Hausdorff measure $\mathcal{H}^1(Z_\lambda)$ of the nodal set has thus been the subject of intense studies over the last three decades, sparked by the well-known conjecture of S.T. Yau (see \cite{Ya1}, \cite{Ya2}) :
\begin{conjecture}\label{conjecture_yau}
Let $(M,g)$ be a compact, $C^\infty$ Riemannian manifold of dimension $n$. There exist positive constants $c$, $C$ such that $$c \loh \leq \mathcal{H}^{n-1}(Z_\lambda) \leq C \loh.$$
\end{conjecture} 
Remark that this paper is concerned with the case $n=2$, but that the conjecture has been stated for smooth manifolds of any dimension. A common intuition in spectral geometry is that a $\lambda$-eigenfunction behaves in many ways similarly to a trigonometric polynomial of degree $\lambda^{\frac{1}{2}}$. As such, one can understand Yau's conjecture as a broad generalization of the fundamental theorem of algebra: counting multiplicities, a polynomial of degree $\lambda^{\frac{1}{2}}$ will vanish $\lambda^{\frac{1}{2}}$ times. 
The conjecture has been proved by Donnelly-Fefferman for real analytic pairs $(M,g)$ of any dimension in \cite{DF1}. When $M$ is a surface with a $C^\infty$ metric, the lower bound was proved by Br{\"u}ning in \cite{Br}. The current best upper bound of $\lambda^\frac{3}{4}$ obtained by \cite{DF2, D} is still weaker than the conjectured one. Note that the current best exponent $\frac{3}{4}$ in dimension $2$ gets much worse in higher dimension. Indeed, for $n \geq 3$, the current best upper bound is $\lambda^{\sqrt{\lambda}}$ and has been obtained by Hardt and Simon in \cite{HS}. This hints that the methods used on surfaces are specific and cannot, in general, be easily extended to higher dimensional manifolds, which is indeed the case for the results of this paper. For more details and a thorough survey of the most recent results on nodal sets of Laplace eigenfunctions, we refer to \cite{Z}.

\subsection{An averaged measure of the local growth.}
Here and elsewhere in this article, given a ball $B(r)$ of radius $r$, $\alpha B$ will denote the concentric ball of radius $\alpha r$. In any metric space, it is possible to measure the growth of a continuous function $f$ by defining its \textit{doubling exponent} $\beta(f, B)$ on a metric ball $B$ by $$\beta(f, B) := \log \frac{\sup_B |f|}{\sup_{\frac{1}{2}B}|f|}.$$ The simplest example is that of the the polynomial $x^n$ on the real interval $D=[-1,1]$, for which the doubling exponent is the degree $n$, modulo a constant. Indeed, $\beta(x^n, [-1,1])  = n \log 2.$ Given two concentric balls $B, \alpha B$, where $0 < \alpha < 1$, one can define the more general $\alpha$-\textit{growth exponent} $\beta(f, B; \alpha)$ by $$\beta(f, B; \alpha) := \log \frac{ \sup_{B} |f| }{ \sup_{\alpha B} |f|}.$$
Albeit more general, the growth exponent can still be seen as the analog of the degree of a polynomial, as showcased once again by the monomial $x^n$ : $$\beta(x^n, [-1,1]; \alpha) = \log \frac{ \sup_{[-1,1]} |x|^n }{ \sup_{[-\alpha, \alpha]} |x|^n} = n \log (\alpha^{-1}).$$
It is worth mentioning that the growth exponent is itself a special case of the more general Bernstein index, which measures in a similar fashion the growth of a continuous function from one compact set to a strictly larger one. For more background on the Bernstein index, we refer to \cite{KY} and \cite{RY}.\\

The metric $g$ turns $M$ into a metric space and it is natural to define similar exponents to measure the growth of eigenfunctions on metric disks on the surface. We write $B_p(r)$ for a metric disk centred at $p \in M$ and of radius $r$. In \cite{DF1}, the authors show that on a smooth manifold $(M,g)$ of any dimension, the following holds for every ball $B$: $$\beta(\phi_\lambda, B) \leq c \loh,$$ where $c = c(g, r, \alpha)$ is a positive constant depending only on the geometry of $M$, the radius $r$ and the scaling factor $\alpha$. From now on, we will restrict our attention to disks $B_p(r)$ of radius comparable to the wavelength: $r = k_0 \lmoh$, where $k_0$ is a suitably small, positive constant. It turns out that, at this scale, the local study of an eigenfunction can be reduced to that of a solution of a planar Schr\"{o}dinger equation (see section \ref{s2p2}), which is a central idea throughout this article. For simplicity, we write $$\beta_p(\lambda) := \beta(\phi_\lambda, B_p(r); \alpha_0)$$ for the $\alpha_0$-growth exponent of $\phi_\lambda$ and where $\alpha_0$ is a geometric constant whose explicit value is given by equation (\ref{def_alpha_not}). The quantity $\beta_p(\lambda)$ is by definition local and, motivated by section 7.3 in \cite{NPS}, we make it global by defining the \textit{average local growth} of a $\lambda$-eigenfunction, which is essentially the averaged $L^1$ norm of $\beta_p(\lambda)$ :
\begin{equation*}
A(\lambda) := \frac{1}{\text{Vol}(M)}\int_M \beta_p(\lambda) \mathrm{d}V_g(p).
\end{equation*}

Thus, $A(\lambda)$ can be interpreted as the expected value of the $\alpha_0$-growth exponent of an eigenfunction $\phi_\lambda$ on disks of wavelength radius.

\subsection{Results.}
We recall the basic intuition of interpreting an eigenfunction $\phi_\lambda$ as a polynomial of degree $\lambda$. In the case of a polynomial, the degree controls both the growth and the number of zeroes and it is thus natural to expect a similar link for eigenfunctions. Our main result proves Conjecture 7.1 of \cite{NPS} and provides such a link by showing that the average local growth is comparable to the size of the nodal set $Z_\lambda$ times the wavelength $\lambda^{-\frac{1}{2}}$.

\begin{main_theorem}\label{thm1}
Let $(M,g)$ be a smooth, closed Riemannian manifold of dimension two. There exist positive constants $c_1, c_2$ such that
\begin{equation}\label{eq_main_thm}
c_1  \lambda^{\frac{1}{2}} A(\lambda) \leq \mathcal{H}^1(Z_\lambda) \leq c_2  \lambda^{\frac{1}{2}} (A(\lambda) + 1).
%c_1 \mathcal{H}^1(Z_\lambda) \lambda^{-\frac{1}{2}} \leq B_1(\lambda) \leq c_2 \mathcal{H}^1(Z_\lambda) \lambda^{-\frac{1}{2}} .
\end{equation}
\end{main_theorem}

The theorem provides an interesting reformulation of Yau's conjecture for surfaces with smooth metric. Recall that in this setting, the lower bound of Conjecture (\ref{conjecture_yau}) is proven, so that, in view of Theorem 1, the conjecture holds if and only if $$A(\lambda) = O(1).$$ Also, since the conjecture is true in the analytic case, we immediately have that $A(\lambda) = O(1)$ in such a setting. In other words, on a surface with a real analytic metric, the average local growth of an eigenfunction on balls of small radius is bounded by a constant independent of the eigenvalue.\\

Finally, two other main results are of interest, namely Theorem \ref{thm2} and Theorem \ref{thm3}, each providing a link between growth exponents and the size of nodal sets of solutions to a planar Schr\"{o}dinger equation. The explicit statement of these results is respectively given at the beginning of sections \ref{s2}, \ref{s3}.\\

\subsection{Outline of proof and organization of the paper}
In section 7.3 of \cite{NPS}, the authors suggested a heuristic for the proof of Theorem \ref{thm1} which essentially consisted of the following 4 steps: 

\begin{itemize}
\item [i.] Reduction of an eigenfunction $\phi_\lambda$ to a solution $F$ of a planar Schr\"{o}dinger equation. This is done locally on a conformal coordinate patch by restricting $\phi_\lambda$ to a small disk of radius $\sim \lmoh$, which transforms the eigenvalue equation (\ref{main_eigenvalue_eq}) into $$\Delta F + qF = 0,$$ where $\Delta$ is the flat Laplacian and $q$ is a smooth potential with small uniform norm. \\
\item [ii.] Use Lemma 3.4 from \cite{NPS} to express $F$ as the composition $u \circ h$ of a harmonic function $u$ with a $K$-quasiconformal homeomorphism $h$ whose dilation factor $K$ is controlled. \\
\item[iii.] Extend to $F$ and then to $\phi_\lambda$ some appropriate estimates linking the size of the nodal set of $u$ with its growth exponent $\beta$. Such estimates are in the spirit of Lemma 2.13 in \cite{NPS} (see also, \cite{G, R, KY} ) and relate the growth exponents of a harmonic function $u$ on some disk with the number of change of signs of $u$ on the boundary of either a larger or a smaller disk. \\
\item[iv.] The final step is an integral-geometric argument based on a generalized Crofton formula that allows to recover the global statement of Theorem \ref{thm1} from the local estimates obtained in the previous steps.
\end{itemize}

This approach has been successful in obtaining the lower bound for the size of the nodal set in terms of the average local growth, that is, the left inequality of Theorem \ref{thm1}.  The details are presented in section \ref{s3}. However, as first noticed by J. Bourgain, the same approach cannot be used for the other inequality. The problem roughly resides in step [iii], where we are aiming to extend to $F = u \circ h$ a result of the type $$N_u(\partial D_-) \leq \beta_u(D^+),$$ where $N_u(\partial D_-)$ is the number of zeros of $u$ on a circle $\partial D_-$ that is strictly contained in a bigger disk $D^+$ on which the doubling exponent is computed. It is impossible to do so, since we have no way to ensure that the $K$-quasiconformal map $h$ will map the circle $\partial D_-$ to another circle in the domain of $F$. It might in fact map a circle to a non-rectifiable curve, which prevents from properly counting the zeros of $F$. \\

Based on a private communication with the authors of [NPS], we take a different route to prove the upper bound in Theorem \ref{thm1}, which is inspired by the work of Donnelly and Fefferman in \cite{DF1}. More precisely, we keep steps [i] and [iv], but replace the intermediate steps by Theorem \ref{thm2}, which provides a convenient estimate linking the size of the nodal set of $F$ on a small disk to its growth exponent on a bigger disk. This approach is presented in section \ref{s2} and allows us to recover the remaining inequality of our main theorem. Theorem \ref{thm2} thus plays a crucial role and its proof is presented in section \ref{s4}. The general idea is to tile the domain of $F$ into squares of rapid and slow growth and to then notice that: a) the nodal set in a square of slow growth is small and b) there can not be too many squares of rapid growth. The interested reader will also find further explanations detailing the structure of that proof in subsection \ref{s4ss2}. Involved in the proof are notably the technical Proposition \ref{prop1}, which roughly proves statement (b) above, as well as the specialized Carleman estimate of Lemma \ref{Carleman}, whose rather long derivations we respectively present in sections \ref{s5}, \ref{s6}. We conclude the article with a discussion and a few questions in section \ref{s7}.\\

\textsc{Notation.} Throughout the paper, we will denote positive numerical constants in the following fashion: $c_1, c_2, c_3, ...$ will be used in the statements of any result and these constants may depend on the geometry of the manifold $M$, but nothing else. In particular, they are independent of $\lambda$. Within proofs, we will use $a_1, a_2, ...$ for numerical constants without any dependency and $b_1, b_2, ...$ for constants that may depend on the geometry of the surface. Often, we merge many numerical constants together to simplify the sometimes heavy notation, for example: $a_5 = a_3^{-1} a_4 \frac{4 \pi}{\text{Vol}(M)}.$ Finally, we reset the numeration for the constants $a_i$ at each section.\\ 

We will use $D$ to denote Euclidean disks and $B$ for metric balls on the surface. Given the context, we either write $D(p,r)$ for a disk centred at $p$ of radius $r$ or just $D_p$ if the radius is known. Finally, we will keep the convention that, given a positive constant $a$ and a disk $D = D(p,r)$, $a D$ denotes the concentric disk of radius $ar$. We write $\mathbb{D}$ for the open unit disk in $\mathbb{R}^2$.\\

\textsc{Acknowledgements.}
This research is part of my Ph.D. thesis at Universit\'e de Montr\'eal under the supervision of Iosif Polterovich. I am very grateful to him for suggesting the problem and for his constant support and many discussions which have been both very helpful and enjoyable. I also want to thank Dan Mangoubi for his support and useful explanations, as well as Leonid Polterovich for helpful remarks. I am also grateful to Steve Zelditch for his suggestions on the exposition as well as some interesting questions. Thanks to Agathe Bray-Bourret for her help with some figures. Finally, I want to specially underline the precious help of Misha Sodin, whose contribution has been more than instrumental in the completion of this article. The main ideas used in the proof of Theorem \ref{thm2} are based on the notes provided by him and I am extremely grateful to have benefited from his support and help.

%-------------------------------------------------------------------------------------------------------
% 
%		SECTION 2 : UPPER BOUND FOR THE LENGTH OF Z_\LAMBDA
%
%--------------------------------------------------------------------------------------------------------

\section{Upper bound for the length of the nodal set}\label{s2}

\stepcounter{subsection}

In this section, we prove the right inequality of Theorem 1, which provides an upper bound to the length of the nodal set in terms of the average local growth of an eigenfunction $\phi_\lambda$. The main tool in the proof is the following result which links the size of the nodal set of a Schr\"{o}dinger eigenfunction to its growth exponent.

\begin{theorem}\label{thm2}
Let $F: 3\mathbb{D} \rightarrow \mathbb{R}$ be a solution of 
\begin{equation}\label{eq_schroedinger_ev_3D}
\Delta F + qF = 0, 
\end{equation}
with the potential $q \in C^{\infty}(3\mathbb{D})$ satisfying $||q||_\infty = \sup_{3\mathbb{D}} |q| < \epsilon_0$. Let also $$\beta := \beta\left(F, \frac{5}{2}\mathbb{D}; 10\right) = \log \frac{\sup_{\frac{5}{2} \mathbb{D} } |F| }{\sup_{\frac{1}{4} \mathbb{D} } |F|}.$$ Finally, denote by $Z_F$ the nodal set $\{ p \in 3 \mathbb{D} : F(p) = 0 \}$ of $F$. Then,
\begin{equation*}
\mathcal{H}^1 \left( Z_F \cap \frac{1}{60}\mathbb{D} \right) \leq c_3 \beta^*,
\end{equation*}
where $\beta^* := \max\{ \beta, 1 \}$ and $c_3$ is a positive constant.
\end{theorem}

We remark that we do not assume here that $q$ has a constant sign. The proof of this theorem is presented in section \ref{s5} and some information about the value of $\epsilon_0$ is given at the end of Lemma $\ref{prop1_lemma6}$.

%-------------------------------------------------------------------------------------------------------
%		SUBSECTION 2.1: FROM THE SURFACE TO THE PLANE
%-------------------------------------------------------------------------------------------------------
\subsection{From the surface to the plane: the passage to Schr\"{o}dinger eigenfunctions with small potential}\label{s2p1}

Cover the surface $M$ with a finite number $N$ of conformal charts $\displaystyle \left(U_i, \psi_i\right)$, $\psi_i: U_i \subset M \rightarrow V_i \subset \mathbb{R}^2$, $i \in I = \{ 1, ..., N \}$. On each of these charts, the metric is conformally flat and there exist smooth positive functions $q_i$ such that  $\displaystyle g = q_i(x,y) (dx^2 + dy^2)$. By compactness, we can find positive constants $q_{-}$ and $q^+$ such that we have $0 < q_{-} <Êq_i < q^+$ for all $i=1, ..., N$. The metric is thus pinched between scalings of the flat metric and we have a local equivalence of various metric notions on $M$ and in $\mathbb{R}^2$. In particular, given any subset $E \subset U_i$, the 1-dimensional Hausdorff measures are equivalent: 
\begin{equation}\label{eq_hausdorff_measures}
b_1 \mathcal{H}^1(\psi_i(E)) \leq \mathcal{H}^1(E) \leq b_2 \mathcal{H}^1(\psi_i(E)).
\end{equation}

In the same spirit, the Riemannian volume form on $M$ and the Lebesgue measure $dA$ in $\mathbb{R}^2$ are equivalent in the following sense: given any integrable function $f$ on $U_i$, we have
\begin{equation}\label{eq_volume_forms}
b_3 \int_{V_i} f \; \mathrm{d}A \leq \int_{U_i} f \; \mathrm{d}V_g \leq b_4 \int_{V_i} f \; \mathrm{d}A.
\end{equation}
Note that the explicit values of the constants $b_1, ..., b_4$ involve only the geometric constants $q_-, q^+$. We now let $B_p := B_p(k_0 \lmoh) \subset M$ be a metric disk and set 
\begin{equation}\label{def_alpha_not}
 \alpha_0 := \frac{q_-}{5q^+}.
 \end{equation} 
The value of the small positive constant $k_0$ will be fixed later. Recall that at a point $p \in M$, the growth exponent $\beta_p(\lambda)$ of an eigenfunction $\phi_\lambda$ is defined by $$\beta_p(\lambda) := \log \frac{ \sup_{B_p} |\phi_\lambda| }{\sup_{\alpha_0 B_p} |\phi_\lambda|}.$$ 

%-------------------------------------------------------------------------------------------------------
%		SUBSECTION 2.2: METRIC AND EUCLIDEAN DISKS 
%-------------------------------------------------------------------------------------------------------
\subsection{Metric and Euclidean disks}\label{s2p2}
In order to estimate $\beta_p(\lambda)$ from below, we define the following Euclidean disks: $$D_p^+ := D_p(q_- k_0 \lmoh), \;\;\; D_p^- := \alpha_0 D_p( q^+ k_0 \lmoh),$$ so that $D_p^-$ is a proper subset of $D_p^+$. Note that by a Euclidean disk $D_p(r)$ centred at $p \in M$, we mean the set $\{ (x,y) : x^2 + y^2 \leq r^2 \}$, where $(x,y)$ are local conformal coordinates around $p$. The inclusions $\displaystyle B_p \supset D_p^+, \;\; \alpha_0 B_p \subset D_p^-$ imply 
\begin{equation*}
\log \frac{ \sup_{D_p^+} |\phi_\lambda| }{ \sup_{D_p^-} |\phi_\lambda| } \leq \beta_p(\lambda).
\end{equation*}

In a conformal chart $\left(U_i, \psi_i\right)$, the eigenvalue equation $\displaystyle \Delta_g \phi_\lambda + \lambda \phi_\lambda = 0$ becomes 
\begin{equation}\label{eq_pre_schr}
 \Delta \phi_\lambda + \lambda q_i \phi_\lambda = 0.
 \end{equation}
 In the aim of using Theorem $\ref{thm2}$, we endow the disk $3\mathbb{D}$ with the complex coordinate $z = x + iy$, fix a scaling constant $\tau = 2q^+ \alpha_0$ and define a function $F = F_{\lambda,p}: 3\mathbb{D} \rightarrow \mathbb{R}$ by $F(z) = \phi_\lambda( \tau k_0 \lmoh z + p).$ The scaling allows us to absorb the spectral parameter $\lambda$ in the potential. Indeed, we have $$\Delta F = \tau^2 k_0^2 \lambda^{-1}\Delta \phi = (k_0 \tau)^2 \lambda^{-1}(-\lambda q_i)\phi_\lambda = -(k_0 \tau)^2 q_i F,$$ so that $F$ satisfies equation (\ref{eq_schroedinger_ev_3D}), where $q = (k_0 \tau)^2 q_i$ is a smooth potential whose supremum norm satisfies $ ||q||_\infty < \epsilon_0$ without loss of generality. Indeed, since the family of $q_i$ is bounded, we can choose $k_0$ as small as needed. The transformation $z \mapsto \tau k_0 \lmoh z + p$ induces the following correspondences between disks in $3\mathbb{D}$ and Euclidean disks centred at $p$: $$ \left\{ |z| \leq \frac{1}{4} \right\} \leftrightarrow D_p^-,\;\; \left\{ |z| \leq \frac{5}{2} \right\} \leftrightarrow D_p^+, \;\; \left\{ |z| \leq \frac{1}{60} \right\} \leftrightarrow D_p^0,$$ where $D_p^0 = D_p\left( \frac{\tau}{60} k_0 \lmoh \right).$ As a consequence, we have 
\begin{equation}\label{Ig_temp1}
\beta^* < \beta + 1 = \log \frac{ \sup_{D_p^+} |\phi_\lambda| }{ \sup_{D_p^-} |\phi_\lambda| } + 1 \leq \beta_p(\lambda) + 1.
\end{equation}

It is important at this stage to remark that the construction of $F$ is dependant on a fixed choice of conformal chart $U_i$, both for the well-posedness of equation (\ref{eq_pre_schr}) as well as the very definition of the Euclidean disks. Thus, in order to allow the construction of $F = F_{\lambda, p}$ everywhere on the surface $M$, one has to choose $k_0$ small enough so that the disks $D_p^* := D_p(3 k_0 \tau \lmoh)$, which are mapped onto $3\mathbb{D}$ are contained in at least one chart $U_i$, for every $p \in M$. This allows the definition of the mapping $\sigma: M \rightarrow I = \left\{1,..., N\right\}$ which assigns to a point $p$ a unique index $\sigma(p)$ such that $D_p^* \subset U_{\sigma(p)}$. The disjoint sets $G_i := \sigma^{-1}(i)$ form a partition of $M$. Figure~\ref{fig_disks_and_balls} summarizes the setting we are in, by presenting a sketch of the various correspondences between Euclidean disks in $G_i$ and those in $3\mathbb{D}$.\\

%-------------------------------------------------
% DBUT FIGURE: METRIC BALLS AND EUCLIDEAN DISKS
%-------------------------------------------------
\begin{figure}[h]
\begin{center}
\includegraphics[scale=0.34]{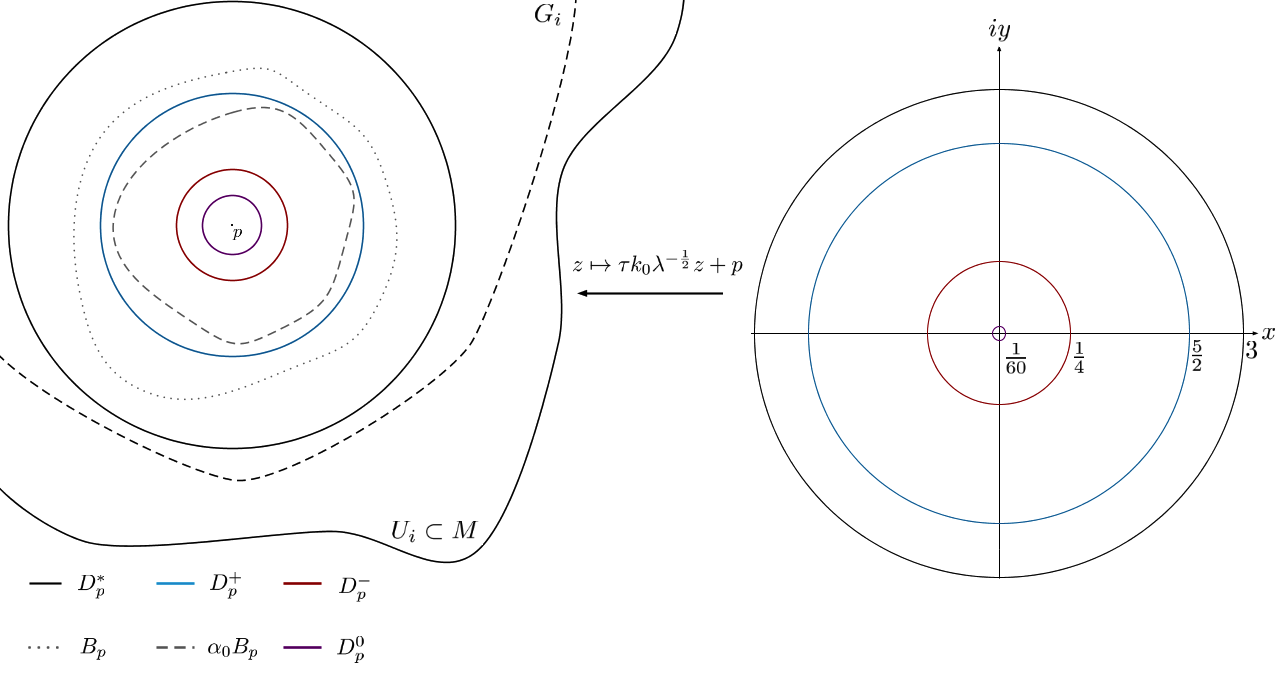}
\caption{Mapping of Euclidean disks and metric balls within a conformal patch.}
\label{fig_disks_and_balls}
\end{center}
\end{figure}
%-------------------------------------------------
% FIN FIGURE
%-------------------------------------------------

We now turn to the study of the nodal set $Z_\lambda$. Recall that $S_\lambda$ is the singular set of the eigenfunction $\phi_\lambda$ and consider the sets $Z_\lambda (i) := \psi_i \left ( ( Z_\lambda \setminus S_\lambda ) \cap G_i\right) \subset \mathbb{R}^2.$ Since $S_\lambda$ is discrete, we have 
\begin{equation}\label{Ig_temp2}
\mathcal{H}^1(Z_\lambda) = \mathcal{H}^1(Z_\lambda \setminus S_\lambda) \leq b_2 \sum_{i \in I} \mathcal{H}^1\left(Z_\lambda(i)\right). 
\end{equation}

Denote by $Z_F$ the nodal set of $F$. By construction, we have
 $$\mathcal{H}^1 (Z_\lambda(i) \cap D_p^0) = (k_0 \tau) \lmoh \mathcal{H}^1\left(Z_F \cap \frac{1}{60}\mathbb{D}\right).$$
Applying Theorem $\ref{thm2}$ and equation ($\ref{Ig_temp1}$) now yields
\begin{equation}\label{Ig_temp3}
\mathcal{H}^1 (Z_\lambda(i) \cap D_p^0) \leq a_2 \lmoh ( \beta_p(\lambda) + 1) .
\end{equation}

We integrate the left-hand side of the last equation over the set $G_i$  and use a generalized Crofton formula (see eq. 6 in \cite{HS}) to get 
\begin{equation}\label{Ig_temp4}
\int_{G_i} \mathcal{H}^1 (Z_\lambda(i) \cap D_p^0) dA(p) = a_3 \mathcal{H}^2(D_p^0) \mathcal{H}^1(Z_\lambda(i)) = a_4 \lambda^{-1}\mathcal{H}^1(Z_\lambda(i)).
\end{equation}

Recalling the equivalence (\ref{eq_volume_forms}) and combining (\ref{Ig_temp3}) and (\ref{Ig_temp4}) then gives
\begin{equation*}
a_4 \lambda^{-1} \mathcal{H}^1(Z_\lambda(i)) \leq a_2 \lmoh \int_{G_i} (\beta_p(\lambda) + 1) dA(p) \leq (a_2 b_3^{-1}) \lmoh \int_{G_i} (\beta_p(\lambda) + 1)  \mathrm{d}V.
\end{equation*}
Simplifying readily gives
\begin{equation*}
\mathcal{H}^1(Z_\lambda(i)) \leq a_5 \lambda^{\frac{1}{2}} \int_{G_i} (\beta_p(\lambda) + 1) \mathrm{d} V,
\end{equation*}
so that
\begin{align*}
\mathcal{H}^1(Z_\lambda) &\leq b_2 \sum_{i \in I} \mathcal{H}^1(Z_\lambda(i)) \leq a_6 \lambda^{\frac{1}{2}}  \sum_{i \in I} \int_{G_i} (\beta_p(\lambda) + 1) \mathrm{d}V = a_6 \lambda^{\frac{1}{2}} \int_M (\beta_p(\lambda) + 1)\mathrm{d} V \\
&\leq c_2 \lambda^{\frac{1}{2}}( A(\lambda) + 1) .
\end{align*}

%-------------------------------------------------------------------------------------------------------
% 
%		SECTION 3 : LOWER BOUND FOR THE LENGTH OF Z_\LAMBDA
%
%--------------------------------------------------------------------------------------------------------

\section{Lower bound for the length of the nodal set}\label{s3}

\stepcounter{subsection}

In this section, we prove the left inequality of Theorem \ref{thm1}. As was the case in the previous section, the central idea is once again the use of conformal coordinates on $M$ and restriction to wavelength scales to reduce the local behaviour of an eigenfunction $\phi_\lambda$ to that of $F$, a solution of a planar Schr\"{o}dinger equation with small, smooth potential. The main result of this section is the following theorem which suitably links the growth exponent of $F$ with its nodal set.

\begin{theorem}\label{thm3}
Let $F: \overline{\mathbb{D}} \rightarrow \mathbb{R}$ be a solution of 
\begin{equation}\label{eq_schroedinger_ev_1D}
\Delta F + qF = 0,
\end{equation}
in $\mathbb{D}$ and with the potential $q \in C^{\infty}(\mathbb{D})$ satisfying $||q||_\infty = \sup_{\overline{\mathbb{D}}} |q| < \epsilon_1$. Denote by $|Z_F(\mathbb{S}^1)|$ the number of zeros of $F$ on the unit circle $\mathbb{S}^1$. Then, 
\begin{equation*}
\log \frac{\sup_{\rho^+ \mathbb{D}}|F|}{\sup_{\rho^-\mathbb{D}}|F|} \leq c_4(1 + |Z_F(\mathbb{S}^1)|),
\end{equation*}
where $0 < \rho^- < \rho^+ < \frac{1}{2}$ are fixed, small radii.
\end{theorem}

The value of $\epsilon_1$ can be obtained in the proof of Lemma 3.3 in [NPS], while those of $\rho^-$ and $ \rho^+$ are given in the proof. The constant $\rho^-$ depends on the geometry of the manifold. It is possible to get rid of this dependancy if one wants Theorem \ref{thm3} to be a stand-alone result. However, our aim is to prove the left inequality of Theorem \ref{thm1} and, as such, our choice of $p^-$ makes the rest of the argument much simpler. Also, remark that, in contrast to Theorem \ref{thm3} where $F$ was defined on $\mathbb{D}$, the setting is now in $3 \mathbb{D}$. This is an arbitrary choice made only in order to ease the writing of the respective proofs: confining Theorem \ref{thm2} to the unit disk would have added even more complexity in the expression of the many constants needed to carry on the long proof.
%-------------------------------------------------------------------------------------------------------
%		SUBSECTION 3.1: PROOF OF THM 3 
%-------------------------------------------------------------------------------------------------------
\subsection{Proof of Theorem \ref{thm3}}

The general strategy is as follows: we first prove a similar kind of result for harmonic functions and, inspired by \cite{NPS}, we then express $F$ as the composition of a harmonic function and a K-quasiconformal homeomorphism. Controlling the properties of the quasiconformal homeomorphism allows to recover the desired result. We begin with a lemma that relates the growth of harmonic functions within a disk and its nodal set on the boundary.

\begin{lemma}\label{lemma_harmonic_fct}
Let $v \in C^\infty (\mathbb{D}) \cap C^0 (\overline{\mathbb{D}})$ be harmonic in the open unit disk and denote by $N_v$ the number of changes of sign of $v$ on the circle $|z| = 1$. Choose $r_0$ in $0 < r_0 < \frac{1}{2}$. Then, 
\begin{equation}
\frac{\sup_{\frac{1}{2}\mathbb{D}}|v|}{\sup_{r_0 \mathbb{D}} |v|} \leq \left(\frac{c_5}{r_0}\right)^{N_v},
\end{equation}
where $c_5$ is  a positive numerical constant.
\end{lemma}

\begin{proof}
Let $u$ be the harmonic conjugate of $v$ such that $u(0) = 0$. Then, the function $$ f(z) = \sum_{n=0}^\infty \xi_n z^n = u(z) + iv(z)$$ is holomorphic in the closed unit disk $\{ |z| \leq 1 \}$.  Suppose that $$\sup_{r_0\mathbb{D}} |v| = \max_{|z| = r_0} |v| = 1.$$ The harmonic function $v$ changes signs $2p= N_v$ times on the circle $|z| = 1$, where $p$ is a non-negative integer. Also, let $\mu_p := \max \{|\xi_0|, |\xi_1|, ..., |\xi_p| \}$. By a result from Robertson (see \cite{R}, Thm. 1, (iii)), we have 
\begin{equation}\label{lemma_robertson_eq1}
|f(re^{i\theta})| < c(p) \mu_p (1-r)^{-2p-1},\;\;\;\; (r < 1), 
\end{equation}
where $c(p) > 0$ is a constant depending on $p$ which will be given explicitly later. Let us remark here that in \cite{R}, the author actually proves (\ref{lemma_robertson_eq1}) in our current setting and then uses a limiting argument to obtain a slightly different statement.\\

The classical Schwarz formula says that for a function $g$ holomorphic on the open disk $r_0 \mathbb{D}$ and continuous on the boundary $\{ |z| = r_0 \}$, we have
\begin{equation*}
g(z) = \frac{1}{2\pi}\int\limits_0^{2\pi}  \text{Re}\left(g(r_0 e^{i\theta})\right) \frac{r_0 e^{i\theta} + z}{r_0 e^{i\theta} - z}\, \mathrm{d}\theta + i\,\text{Im}(g(0)), \;\; |z| < r_0.
\end{equation*}

Since $f = u + iv$ is holomorphic, so is $g = v - iu$ and we obviously have $|f| = |g|$, so that the following inequality holds for all $\displaystyle |z| \leq \frac{r_0}{2}$:
\begin{align*}
|f(z)| &= |g(z)| = \left| \frac{1}{2\pi}\int\limits_0^{2\pi}  \text{Re}\left(g(r_0 e^{i\theta})\right) \frac{r_0 e^{i\theta} + z}{r_0 e^{i\theta} - z}\, \mathrm{d}\theta + i\,u(0) \right| \\
&\leq \left(\frac{1}{2\pi}\right) \max_{|z| = r_0} |v| \int\limits_0^{2\pi}  \frac{r_0 + |z|}{\left| r_0 - z\right|} \, \mathrm{d}\theta \leq 3 =: a_1.
\end{align*}

Applying Cauchy's inequality for holomorphic functions to $\displaystyle f = \sum_{n=0}^\infty \xi_n z^n$ on the open disk of radius $\displaystyle \frac{r_0}{2}$ gives
\begin{equation*}
|\xi_n| \leq \left( \frac{r_0}{2} \right) ^{-n} \sup_{|z| = \frac{r_0}{2}} |f(z)| = a_1 \left(\frac{2}{r_0}\right)^n.
\end{equation*}

Hence, we have $\mu_p \leq a_1 \left(\frac{2}{r_0} \right)^p$. Setting $r = \frac{1}{2}$ in equation (\ref{lemma_robertson_eq1}) now yields

\begin{equation*}
\left|f\left( \frac{1}{2}e^{i\theta}\right)\right| \leq c(p)\, \mu_p\, 2^{2p+1} \leq 2\, a_1\, c(p) \left(\frac{2}{r_0} \right)^p 2^{2p} \leq 2\, a_1\, c(p) \left(\frac{4}{r_0} \right)^{2p},
\end{equation*}
which in turn means $$\sup_{\frac{1}{2} \mathbb{D}} |v|  = \max_{|z| = \frac{1}{2}} |v| \leq \left|f\left( \frac{1}{2}e^{i\theta}\right)\right| \leq 2\, a_1\, c(p) \left(\frac{4}{r_0} \right)^{2p}.$$ Going back to \cite{R}, we use the explicit value of the constant $c(p)$ to get the following bound $$c(p) = 2^{2p} + \frac{(2p)!}{(p!)^2} = 2^{2p} + {2p \choose p} \leq 2^{2p} + \left( \frac{2pe}{p}\right)^p = 2^{2p} +(2e)^p \leq 2  (2e)^{2p}.$$

Since we assumed that $\sup_{r_0 \mathbb{D}} |v| = 1$, we have $$\frac{\sup_{\frac{1}{2} \mathbb{D}} |v|}{\sup_{r_0 \mathbb{D}} |v|}  \leq 4 a_1 \left(\frac{8e}{r_0}\right)^{2p}.$$ 

Suppose now that $\sup_{r_0 \mathbb{D}} |v| = \tau \neq 1$ and let as before $f = u + iv$ be the holomorphic function built from $v$ and its harmonic conjugate $u$. Define $\tilde{f} = \tilde{u} + i\tilde{v}$ by $\tilde{f} = \tau^{-1}f$. Then, $\sup_{r_0 \mathbb{D}} |\tilde{v}| = 1$ and $$ \frac{\sup_{\frac{1}{2} \mathbb{D}} |v|}{\sup_{r_0 \mathbb{D}} |v|}  = \frac{\tau\sup_{\frac{1}{2} \mathbb{D}} |\tilde{v}|}{\tau\sup_{r_0 \mathbb{D}} |\tilde{v}|}  \leq 4 a_1 \left(\frac{8e}{r_0}\right)^{2p} \leq \left( \frac{c_5}{r_0}\right)^{2p}.$$

\end{proof}

We now prove Theorem \ref{thm3}. By Lemmas 3.3 and 3.4 in \cite{NPS}, there exist a $K$-quasiconformal homeomorphism $h: \mathbb{D} \rightarrow \mathbb{D}$ with $h(0) = 0$, a harmonic function $v: \mathbb{D} \rightarrow \mathbb{R}$ and a solution $\varphi$ to equation (\ref{eq_schroedinger_ev_1D}) such that $F = \varphi \cdot (v \circ h)$. Moreover, the function $\varphi$ is positive and satisfies $$1 - a_2\epsilon_1 \leq \varphi \leq 1.$$ Finally, the dilation factor of the quasiconformal map $h$ satisfies $$1 \leq K \leq 1 + a_3 ||q||_\infty \leq a_4.$$ We refer the reader to \cite{NPS} for the precise values of the various constants stated above. We recall Mori's theorem (see section IIIC in \cite{A} or \cite{NPS}) for $K$-quasiconformal homeomorphisms: $$\frac{1}{16}|z_1 - z_2|^K \leq |h(z_1) - h(z_2)| \leq 16|z_1 - z_2|^{\frac{1}{K}}.$$
Since the origin is a fixed point of $h$, we have $$\frac{1}{16}|z|^K \leq |h(z)| \leq 16|z|^{\frac{1}{K}},\; z \in \mathbb{D}.$$ Fix a small radius $\displaystyle \rho^+ = \left( \frac{1}{32}\right)^{a_4}$ and consider the circle $\left\{|z| = \rho^+ \right\}$. For such $z$, Mori's theorem gives $\displaystyle |h(z)| \leq 16 (\rho^+)^{\frac{1}{K}} \leq \frac{1}{2}$ so that  $$h\left( \rho^+ \mathbb{D} \right) \subset \frac{1}{2}\mathbb{D}.$$ Now, set $\rho^- := \frac{\rho^+}{5}\left(\frac{q_-}{q^+}\right)^2$. The image by $h$ of the circle $\left\{ |z| = \rho^- \right\}$ contains the circle of radius $\displaystyle \frac{1}{16} (\rho^-)^K \geq \frac{1}{16}(\rho^-)^{a_4} =: r_0$. As a consequence, we have $$ r_0 \mathbb{D} \subset h\left((\rho^-) \mathbb{D}\right).$$  Since $F = \varphi \cdot ( v \circ h)$, the bounds on $\varphi$ and the above inclusions imply 
\begin{equation*}
\frac{\sup_{\rho^+ \mathbb{D}} |F|}{\sup_{\rho^- \mathbb{D}} |F|} \leq a_5 \frac{\sup_{\rho^+ \mathbb{D}} |v \circ h|}{\sup_{\rho^- \mathbb{D}} |v \circ h|} \leq  a_5\frac{\sup_{\frac{1}{2} \mathbb{D}} |v|}{\sup_{r_0 \mathbb{D}} |v|},
\end{equation*}
where $a_5 = (1 - a_2 \epsilon_1)^{-1}.$ Since $\varphi$ is positive and $h$ is a homeomorphism, the number $N_F$ of sign changes of $F$ on the unit circle is the same that of $v$. Applying Lemma \ref{lemma_harmonic_fct} now yields
\begin{equation*}
\frac{\sup_{\rho^+ \mathbb{D}} |F|}{\sup_{\rho^- \mathbb{D}} |F|} \leq a_5 \left(\frac{c_5}{r_0}\right)^{N_F}.
\end{equation*}
Since the number $|Z_F(\mathbb{S}^1)|$ of zeros of $F$ on the unit circle is bounded below by $N_F$, taking the logarithm on both sides yields
\begin{equation*}
\log \frac{\sup_{\rho^+ \mathbb{D}} |F|}{\sup_{\rho^- \mathbb{D}} |F|} \leq c_4( 1 + |Z_F(\mathbb{S}^1)|).
\end{equation*}
where $c_4 = \max \left\{a_5, \frac{c_5}{r_0}\right\}.$

%-------------------------------------------------------------------------------------------------------
%		SUBSECTION 3.2: 
%-------------------------------------------------------------------------------------------------------
\subsection{A lower bound for the nodal set in terms of the average local growth.}

In order to recover the right inequality of Theorem \ref{thm1}, we propose an argument which is very similar to the one developed in section \ref{s2}. It thus helps to refer to that section when reading the remainder of this one. The aim is to apply Theorem \ref{thm3} to a function $F$ which has been built from an eigenfunction $\phi_\lambda$ and to then apply an integral geometric argument to recover the desired result. We begin with the same setting as that of Subsection \ref{s2p1} and then define the following Euclidean disks:  $$D_p^+ := D_p(q^+ k_0 \lmoh),\;\;\; D_p^- := \alpha_0 D_p\left( q^- k_0 \lmoh \right).$$ Remark that the last two definitions employ the same notation as in the previous section but the radii of the disks are different. The inclusions $B_p \subset D_p^+$ and $\alpha B_p \supset D_p^-$ imply 
\begin{equation}\label{sec3_tmp_eq1}
\beta_p(\lambda) \leq \log \frac{\sup_{D_p^+} |\phi_\lambda|}{\sup_{D_p^-} |\phi_\lambda|}.
\end{equation}
Let $\tau := \frac{q^+}{\rho^+}$ be a scaling constant, endow the unit disk with the complex coordinate $z = x + iy$ and define $F_{\lambda, p} = F: \mathbb{D} \rightarrow \mathbb{R}$ by $F(z) = F( \tau k_0 \lmoh z + p)$. The function $F$ solves equation (\ref{eq_schroedinger_ev_1D}) and the potential $q$ satisfies $||q||_\infty < \min\{\epsilon_0, \epsilon_1\}$ without loss of generality, choosing $k_0$ small enough. Recalling that $\displaystyle \rho^- = \frac{\rho^+}{5}\left( \frac{q^-}{q^+} \right)^2$, we remark that the mapping $z \mapsto \tau k_0 \lmoh z + p$ induces the following bijections: $$ \left\{ |z| \leq \rho^+ \right\}Ê\leftrightarrow D_p^+, \;\; \left\{ |z| \leq \rho^- \right\}Ê\leftrightarrow D_p^-.$$

An immediate consequence is 
\begin{equation}\label{sec_3_temp2}
 \log \frac{ \sup_{\rho^+ \mathbb{D}} |F| }{ \sup_{\rho^- \mathbb{D} } |F| }   = \log \frac{\sup_{D_p^+} |\phi_\lambda|}{\sup_{D_p^-} |\phi_\lambda|} \geq \beta_p(\lambda).
\end{equation}

Notice that for $F$ to be properly defined on $\mathbb{D}$, the Euclidean disk $D^0_p := D_p(\tau k_0 \lmoh)$ must lie completely within some conformal chart $U_i$. Hence, to ensure that the above construction can be carried through for any $p \in M$, we choose $k_0$ small enough that $D_p(\tau k_0 \lmoh)$ is a proper subset of at least one conformal chart $U_i$ for every $p \in M$. This allows to define the map $\sigma: M \rightarrow I = \{1, ..., N\}$ which assigns to $p \in M$ a unique index $\sigma(p)$ such that $D_p(\tau k_0 \lmoh) \subset U_{\sigma(p)}$. Once again, the sets $G_i := \sigma^{-1}(i) \subset U_i$ form a partition of $M$. Now, consider the sets $Z_\lambda (i) := \psi_i \left (  (Z_\lambda \setminus S_\lambda)  \cap G_i\right), i = 1, ..., N$. Then,
\begin{equation}\label{Ig_sec_3_temp2}
\mathcal{H}^1(Z_\lambda)  = \mathcal{H}^1(Z_\lambda \setminus S_\lambda) \geq b_1 \sum_{i \in I} \mathcal{H}^1\left(Z_\lambda(i)\right). 
\end{equation}

Denote by $|Z_{p,\lambda}(i)|$ the number of intersection points of the circle $\partial D_p^0$ with $Z_\lambda(i)$. By construction, the following equality holds outside from the singular set, that is, almost everywhere:
\begin{equation}\label{Ig_sec_3_temp3}
|Z_{p,\lambda}(i)| = |Z_F( \mathbb{S}^1) |.
\end{equation}

Applying Theorem \ref{thm3} and equation ($\ref{sec_3_temp2}$) now yields
\begin{equation}\label{Ig_sec_3_temp3}
\beta_p(\lambda) \leq c_4( 1 + |Z_{p,\lambda}(i)|,
\end{equation}
outside from $S_\lambda$. We integrate the left-hand side of the last equation over the set $G_i$  and use a generalized Crofton formula (see \cite{HS}, eq. 6) to get 
\begin{equation}\label{Ig_sec_3_temp4}
\int_{G_i \setminus S_\lambda } |Z_\lambda(\partial D_p^0)| dA(p) = a_2 \mathcal{H}^1(\partial D_p^0) \mathcal{H}^1(Z_\lambda(i)) = a_3 \lmoh \mathcal{H}^1(Z_\lambda(i)).
\end{equation}

Notice that, in contrast with the previous use of an analog Crofton formula in section \ref{s2}, we have now integrated,  over all planar rigid motions, the cardinality of the intersection of a one dimensional rotation invariant submanifold - namely the circle $\partial D_p^0$ - with the one dimensional nodal set.\\

It is now straightforward to conclude:
\begin{align*}
A(\lambda) &= \frac{1}{\text{Vol}(M)} \sum_{i \in I} \int_{G_i \setminus S_\lambda} \beta_p(\lambda) \mathrm{d} V_g \\
&\leq (b_4 c_4)\left(1 + \frac{1}{\text{Vol}(M)} \sum_{i \in I} \int_{G_i \setminus S_\lambda}  |Z_{p,\lambda}(i)| dA(p) \right) \\
&= a_5\left(1 + \frac{a_3}{\text{Vol}(M)} \lmoh \sum_{i \in I} \mathcal{H}^1(Z_\lambda(i))\right) \\
&\leq a_6(1 + \lmoh \mathcal{H}^1(Z_\lambda)) \\
&\leq c_1 \mathcal{H}^1(Z_\lambda) \lmoh,
\end{align*}
where the last inequality uses the fact that the lower bound in Yau's conjecture holds for surfaces, preventing $ \lmoh \mathcal{H}^1(Z_\lambda) $ to be too small.

%-----------------------------------------------------------------------------------------------------------------
% 
%		SECTION 4 : NODAL SET AND GROWTH OF PLANE SCHRDINGER EF
%
%-----------------------------------------------------------------------------------------------------------------

\section{Nodal set and growth of planar Schr\"{o}dinger eigenfunctions with small potential}\label{s4}
This section is dedicated to the proof of Theorem $\ref{thm2}$. We start with a function $F: 3\mathbb{D} \rightarrow \mathbb{R}$ which satisfies the equation $\displaystyle \Delta F + qF = 0$ on $3\mathbb{D}$. The potential $q$ is smooth and has a small uniform norm: $||q||_\infty < \epsilon_0$.  Recall that $$\beta = \log \frac{\sup_{\frac{5}{2} \mathbb{D} } |F| }{\sup_{\frac{1}{4} \mathbb{D} } |F|},$$ and that $\beta^* = \max\{ \beta, 1 \}$.

\subsection{A configuration of disks and annuli.}\label{s4ss1}
We start with some notation for disks and annuli within our main setting which takes place in the disk $3\mathbb{D}$. We denote a finite set of \textit{small disks} by $$D_\nu = D(z_\nu, \delta) \subset \frac{1}{60}\mathbb{D},\;\; 1 \leq \nu \leq N,$$ and where the radius $\delta > 0$ is suitably small.
We will say that such a set of small disks is $\gamma$-\textit{separated} if it satisfies: $| z_\mu - z_\nu | \geq 2\gamma \delta$, for all $\mu \neq \nu$ and where $\gamma$ is some positive constant. One has to understand the $\gamma$-separation condition as disjointness after a scaling of factor $\gamma$. For example, in in Figure~\ref{fig_disks_scaled}, the disks $D_1$ and $D_2$ are $\gamma$-separated while the pair $D_\nu$ and $D_N$ is not.\\

%-------------------------------------------------
% DBUT FIGURE: CERCLES
%-------------------------------------------------
\begin{figure}[h]
\begin{center}
\includegraphics[scale=0.45]{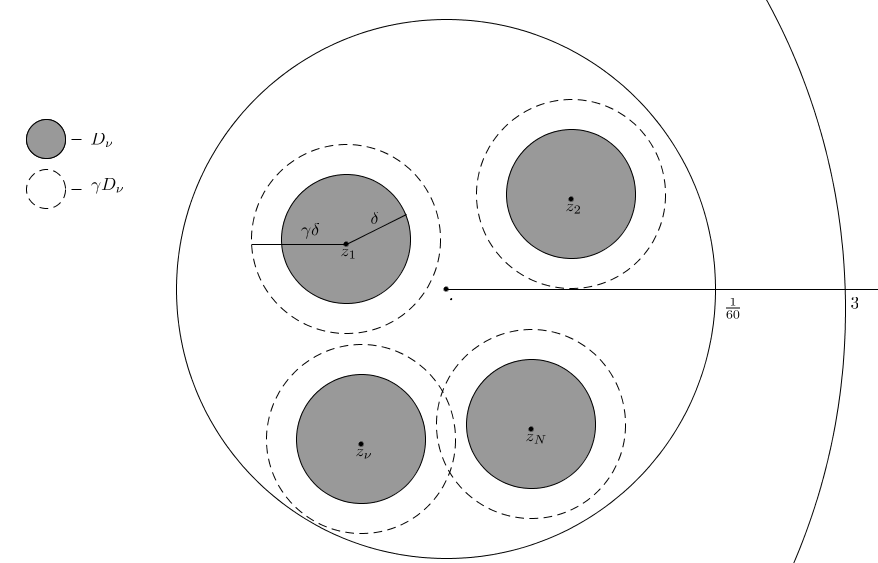}
\caption{A finite set of disks $D_\nu$ and scaled disks within $\frac{1}{60}\mathbb{D}$}
\label{fig_disks_scaled}
\end{center}
\end{figure}
%-------------------------------------------------
% FIN FIGURE
%-------------------------------------------------

For a small $0 < a \ll 1$, we now let  $D_\nu(a) := (1-2a)D_\nu$ and define the following annuli:
\begin{itemize}
\item $A_\nu = \left\{ (1-2a)\delta < | z - z_\nu | < (1-a)\delta \right\}$,\\
\item $A_{\nu'} = \left\{ (1-3a)\delta < | z - z_\nu | < \left(1-\frac{4}{3a}\right) \delta \right\}$,\\
\item $A_{\nu''} = \left\{ \left( 1-\frac{3}{2}a \right) \delta < | z - z_\nu | < (1-a) \delta \right\}$.
\end{itemize}

We regroup the collection of annuli $A_\nu$ under $A = \bigcup\limits_\nu A_\nu.$ Figure~\ref{fig_annuli} provides a close-up of the various annuli defined above.  \\

%-------------------------------------------------
% DBUT FIGURE: ANNEAUX
%-------------------------------------------------
\begin{figure}[h]
\begin{center}
\includegraphics[scale=0.32]{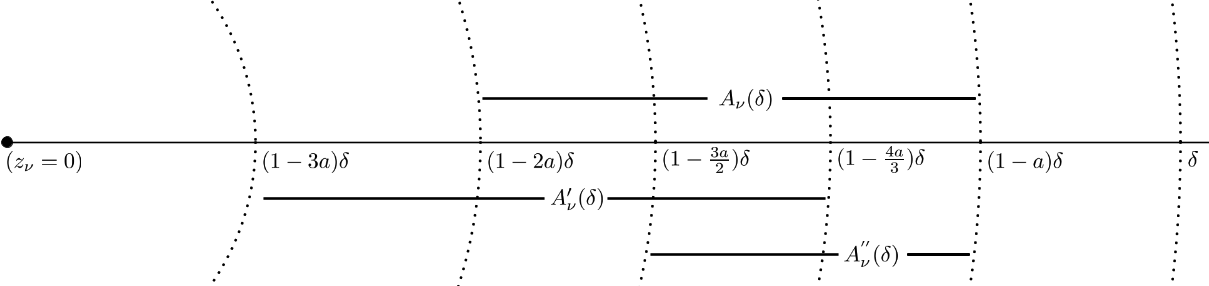}
\caption{Various annuli within a disk $D_\nu$ of radius $\delta$ centred in $z_\nu$}
\label{fig_annuli}
\end{center}
\end{figure}
%-------------------------------------------------
% FIN FIGURE
%-------------------------------------------------

Given $M > 0$, we say that a disk $D_\nu$ is a \textit{disk of M-rapid growth} or simply a \textit{rapid disk} if 
\begin{equation}\label{eq_M_rapid_growth}
M \int_{A_{\nu'}} F^2 \leq \int_{A_{\nu''}} F^2.
\end{equation}

We say the radius $\delta$ is $\beta^*$-related if it satisfies 
\begin{equation}\label{delta_constraint}
\delta < \frac{1}{60},\;\;\; \delta \beta^* < \frac{1}{2}.
\end{equation}
Finally, we fix the separation constant to $\gamma := \delta^{-\frac{1}{2}}$.

\subsection{Intermediate results}\label{s4ss2}
We first state a result that shows that if the potential is small enough and if we fix the growth threshold $M$ sufficiently high, there can not be too many disks of rapid growth. In fact, it turns out that the number of such disks is bounded above by a constant times the growth exponent $\beta^*$: 

\begin{proposition} \label{prop1}
Suppose that the radius of a collection of $\gamma$-separated small disks in $\frac{1}{60}\mathbb{D}$ satisfies the constraints (\ref{delta_constraint}) and let $\mathcal{N} = \mathcal{N}(M)$ denote the number of such disks which are of $M$-rapid growth. Then, $$ \mathcal{N} \leq c_5 \beta^*,$$
provided that $||q||_{\infty} < \epsilon_0$ and $M > M_0$, where $c_5, \epsilon_0, M_0$ are positive constants.
\end{proposition}

The rather long proof, inspired from that of Proposition 4.7 in \cite{DF2}, is presented in section \ref{s5}. The next result is Proposition 5.14 in \cite{DF2} and links the growth condition and the local length of the nodal set.

\begin{proposition}\label{prop2}
Suppose that the disk of radius $\epsilon$ centred in $z_\mu$ is not $M_0$-rapid, that is $$ \int_{\left(1- \frac{3}{2}a\right)\epsilon < |z - z_\mu| < (1-a)\epsilon } F^2  \leq M^{-1} \int_{(1-3a)\epsilon < |z - z_\mu| < \left(1 - \frac{4}{3}a \right)\epsilon} F^2 $$ holds. Then, $$\mathcal{H}^1\left(Z_F \cap D \left(z_\mu, c_6 \epsilon\right) \right) \leq c_7 \epsilon,$$where $c_6, c_7 > 0$ are positive constants.
\end{proposition}

The last two propositions allow us to lay out a general strategy to prove Theorem \ref{thm2}. Indeed, we now know that: (i) there cannot be too many disks of rapid growth and (ii) the nodal set of a slow disk cannot be too big. Conjugating those two ideas in the right way will allow us to bound the global length of the nodal set by the the growth exponent of $F$. 

\bigskip

The proof is based on an iterative process that will be indexed by $k=0,1,2,...$ We begin the first step $k=0$ by fixing some $\delta(0)$ satisfying the constraints ($\ref{delta_constraint}$) and then divide the square $P = \left\{ (x,y) : |x|, |y| < \frac{1}{60} \right\}$ into a grid of squares whose sides have length $\delta(0)$. We distribute those smaller squares into two categories. The \textit{rapid squares} $R_i(0)$, $i=1,2,..., r(0),$ are those which contain at least one point $z_i(0) \in R_i(0)$ such that $D_i = D(z_i, \delta)$ is a disk of $M$-rapid growth of the function $F$. Here, we have fixed $M=M_0$ to allow the use of Proposition \ref{prop1}. If that condition is not satisfied, we consider the square to be a \textit{slow square} and label it $S_j(0)$, $j=1,2,,..., s(0)$. \\ 

We now proceed to the next step $k=1$ and set $\displaystyle \delta(1) = \frac{\delta(0)}{2}.$ We bisect the rapid squares $R_i(0)$ of the previous step into 4 smaller squares and split those newly obtained squares into rapid squares $R_i(1)$, $i=1,2,...,r(1)$ and slow squares $S_j(1),\,j=1,2,...,s(1)$ depending on whether or not they include a point which is the centre of a $M$-rapid disk of radius $\delta(1)$. Note that the slow squares of the previous step are left untouched. Figure~\ref{fig_tiled_squares} gives a representation of the tiling process.\\
%-------------------------------------------------
% DBUT FIGURE: TILED SQUARES
%-------------------------------------------------
\begin{figure}[h]
\begin{center}
\includegraphics[scale=0.58]{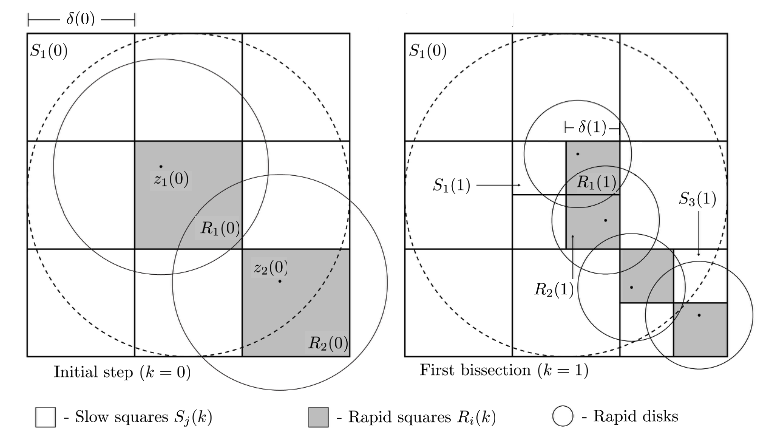}
\caption{Iterative tiling of P in rapid and slow squares. }
\label{fig_tiled_squares}
\end{center}
\end{figure}
%-------------------------------------------------
% FIN FIGURE
%-------------------------------------------------
\\We repeat the process so that, at step k, we have $\delta(k) = 2^{-k}\delta(0)$ as well as some rapid squares $R_i(k)$ and slow squares $S_j(k)$. Let $I(k) = \{1, 2, ... , r(k)\}$ be the indexing set of the rapid squares obtained at step $k$. To simplify notation, we will sometimes write $\delta$ instead of $\delta(0)$ in what follows and until the end of the section.

\begin{lemma}\label{thmSchrod_lemma1}
Denote by $\left| I(k)Ê\right|$ the cardinality of the finite set $I(k)$, e.g. the number of rapid squares at step $k$. There exists a constant $c_8 > 0$ such that, for each step $k=0,1,2...$, we have $$ \left| I(k) \right| \leq c_8 \delta^{-1} \beta^*.$$
\end{lemma}

\begin{proof}

Recall that $\delta(k) := 2^{-k}\delta(0)$. Since $\delta(0)$ satisfies the constraints ($\ref{delta_constraint}$), it follows that $\delta(k)$ is $\beta^*$-related, for all $k \in \mathbb{N} \cup \{0\}$.\\ 

We choose some $\nu \in I(k)$ and recall that there is one rapid growth disk $D_\nu(k)$ whose centre $z_\nu$ lies in $R_\nu(k)$. Notice that, since $\gamma \delta(k) > \sqrt{2}\delta(k)$, we have $R_\nu(k) \subset \gamma D_\nu (k)$, as shown in  Figure~\ref{fig_rapid_sq}.\\

%-------------------------------------------------
% DBUT FIGURE: ZOOM RAPID SQUARE
%-------------------------------------------------
\begin{figure}[h]
\begin{center}
\includegraphics[scale=0.45]{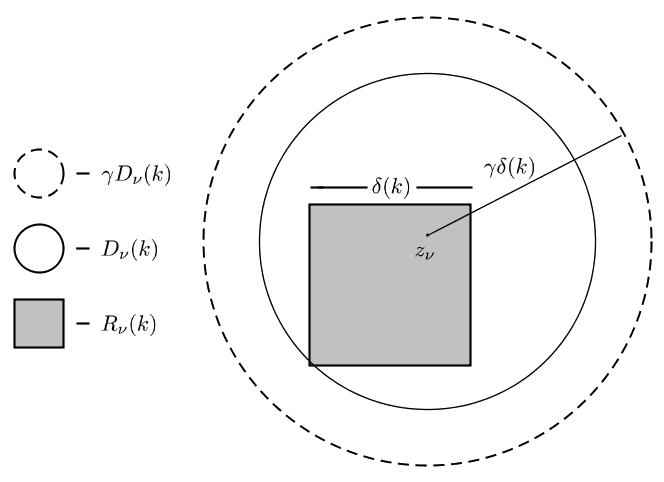}
\caption{A close-up of a rapid square. }
\label{fig_rapid_sq}
\end{center}
\end{figure}
%-------------------------------------------------
% FIN FIGURE
%-------------------------------------------------

Thus, we have
\begin{equation*}\label{thmSchrod_lemma1_eq1}
\bigcup\limits_{\nu \in I(k)} R_\nu (k) \subset \bigcup\limits_{\nu \in I(k)} \gamma D_\nu(k).
\end{equation*}

We now choose a \textit{maximal} subcollection of disjoint disks $\gamma D_\nu$ and denote by $I^*(k) \subset I(k)$ the corresponding set of indices. Notice that disjointness of two scaled disks $\gamma D_\nu, \gamma D_\mu, $ is equivalent to $\gamma$-separation of $D_\nu$ and $D_\mu$. By maximality, for $\mu \in I(k) \setminus I^*(k)$, there exists $\nu \in I^*(k)$ such that $\displaystyle \left| z_\mu - z_\nu  \right| \leq 2 \gamma \delta(k)$. In such a case and for all $z \in \gamma D_\mu (k)$, we thus have $$ \left| z - z_\nu \right| \leq \left| z - z_\mu \right| + \left| z_\mu - z_\nu \right| \leq \gamma \delta(k) + 2 \gamma \delta(k) < 4\gamma \delta(k). $$
As a consequence, we get the following inclusion: $\displaystyle \gamma D_\mu (k) \subset 4 \gamma D_\nu(k)$, where $\mu$ represents a disk excluded from the maximal subset. This in turn means 
\begin{equation*}\label{thmSchrod_lemma1_eq2}
\bigcup\limits_{\nu \in I(k)} \gamma D_\nu(k) \subset \bigcup\limits_{\nu \in I^*(k)} 4 \gamma D_\nu(k).
\end{equation*}
Hence, 
\begin{equation*}
\bigcup\limits_{\nu \in I(k)} R_\nu (k) \subset \bigcup\limits_{\nu \in I^*(k)} 4 \gamma D_\nu(k).
\end{equation*}
We compare the respective areas of the regions covered by the last inclusion and get $\left| I(k) \right| \delta^2(k) \leq 16 \pi \gamma^2 \delta^2(k) \left | I^*(k) \right |$. By Proposition \ref{prop1}, $\left| I^*(k) \right| \leq c_5 \beta^*$ and we finally get $$\left| I(k) \right| \leq 16 \pi \gamma^2 \left| I^*(k) \right | \leq 16 \pi c_5  \gamma^2 \beta^* = c_8 \delta^{-1}(0) \beta^*,$$
which concludes the proof since $I$ is precisely the set indexing the rapid squares.
\end{proof}

\begin{lemma}\label{thmSchrod_lemma2}
Denote by $\left| J(k) \right|$ the number of slow squares $S_j(k)$ obtained at step $k$. Then, for any $k=0,1,2...$, we have $$\left| J(k) \right| \leq 4 c_8 \delta^{-1} \beta^*.$$ 
\end{lemma}

\begin{proof}
By construction, we have: $\displaystyle \left | J(k) \right | \leq 4 \left | I(k-1) \right| \leq 4 c_8 \delta^{-1} \beta^*.$
\end{proof}

\begin{lemma}\label{thmSchrod_lemma3}
There exists a constant $c_9$ such that, for each slow square $S_j(k)$ and each $k=0,1,2...$, we have 
$$\mathcal{H}^1\left( Z_F \cap S_j(k) \right)  \leq c_9 2^{-k}\delta.$$
\end{lemma}

\begin{proof}
If $z_\mu$ lies in some slow square $S_i(k)$, then the disk $D\left(z_\mu, \delta(k)\right)$ is slow, which means it satisfies $$\int_{(1-3a)\delta < |z - z_\mu| < (1-\frac{4}{3}a)\delta} F^2 > M^{-1} \int_{(1-\frac{3}{2}a)\delta < |z - z_\mu| < (1-a)\delta} F^2.$$

By Proposition \ref{prop2}, we thus have $$ \mathcal{H}^1 \left( Z_F \cap \mathcal{D}\left(z_\mu, c_6 2^{-k}\delta \right)  \right)  \leq c_7 2^{-k} \delta,$$
which holds for all $z_\mu \in S_j(k).$ We can now pick a finite collection of $N_0 =N_0(c_6)$ points $z_l \in S_j(k)$ such that the reunion of the associated disks $D \left(z_l, c_6  2^{-k}\delta\right)$ cover $S_j(k)$. The collection being finite, we have
\begin{align*}
\mathcal{H}^1\left( Z_F \cap S_j(k)\right) &\leq \sum\limits_{l=1}^{N_0} \mathcal{H}^1 \left( Z_F \cap D\left(z_l, c_6 2^{-k} \delta\right) \right) \\
&\leq (N_0 c_7) 2^{-k}\delta = c_9 2^{-k} \delta.
\end{align*}
\end{proof}

The next result is exactly Lemma 6.3 in \cite{DF2}.
\begin{lemma}\label{thmSchrod_lemma4}
The union $\displaystyle \bigcup_{j \in J(k),k \in \mathbb{N} \cup \{0\} } S_j(k)$ covers the whole square $$P = \left\{ (|x|,|y| \leq \frac{1}{60} \right\},$$ except for the singular set $\mathcal{S}_F := \left\{ z \in P : F(z) = \nabla F(z) = 0 \right\}$.
\end{lemma}

The last lemma allows us to discard the singular set when studying the length of the nodal set of $F$.

\begin{lemma}\label{thmSchrod_lemma5}
Let $\mathcal{S}_F$ be the singular set of $F$ in $P$. Then, $$\mathcal{H}^1\left(\mathcal{S}\right) = 0.$$
\end{lemma}

\begin{proof}
It is well known (see for instance \cite{B, HL}) that the singular set $\mathcal{S}$ of a $F$ is a submanifold of codimension 2, which means here that it is a finite set of points, whence $\mathcal{H}^1\left(\mathcal{S}\right) = 0$.
\end{proof}

We are now ready to complete the proof of Theorem \ref{thm2}. 

\begin{proof}

Using all of the above lemmas, we have:
\begin{align*}
\mathcal{H}^1\left( Z_F \cap \frac{1}{60} \mathbb{D} \right) &\leq \mathcal{H}^1 \left( Z_F \cap P \right) \underset{6,7}{=} \sum_{k=0}^\infty \sum\limits_{j \in J(k)} \mathcal{H}^1\left( Z_F \cap S_j(k)\right)\\ &\underset{5}{\leq} c_9 \delta \sum_{k=0}^\infty \sum\limits_{j \in J(k)} 2^{-k}
\underset{4}{\leq} (c_9 \delta) 4 c_8 \delta^{-1}\beta^* \sum_{k=0}^\infty 2^{-k} \\
&= 4 c_8 c_9 \beta^* \sum_{k=0}^\infty 2^{-k} \leq c_3 \beta^*.
\end{align*}

\end{proof}

%-----------------------------------------------------------------------------------------------------------------
% 
%		SECTION 5 : PROOF OF PROP 1
%
%-----------------------------------------------------------------------------------------------------------------

\section{Proof of Proposition 1}\label{s5}

We divide the rather long proof in $6$ subsections. The treatment is based on the proof of Proposition 4.7 in \cite {DF2}.
%\textsc{\underline{1. Setting.}}\\
\subsection{Setting}

%OLD
%For simplicity, we set the scaling factor to $\gamma := \delta^{-\frac{2}{3}}$. Using the same hypotheses, we will actually prove a slightly more general statement. We first give %ourselves a bit more room by introducing a small parameter $\xi > 0$ and we then replace the growth exponent $\beta$ by $t := (\beta + 1) + \xi$. Now let $F: 3\mathbb{D} \rightarrow \mathbb{R}$ be our solution to 
%\begin{equation*}
%\Delta F + \xi \tilde{q} F = 0,
%\end{equation*} 
%which is just another form of Equation (\ref{eq_schroedinger_ev_3D}) with $\tilde{q} = \xi^{-1}q$. We can pick $0 < \xi < 1$ small enough so that the parameters $(\gamma, \delta)$ %are $t$-related, that is
%\begin{equation}\label{relation_delta_t}
%\gamma \geq \delta^{-\frac{2}{3}}, \;\;\; \delta t \leq 1. 
%\end{equation} 
%Finally, we normalize $F$ by the condition $\sup_{3\mathbb{D}} |F| = 1$, which has no effect whatsoever on the growth exponent. We will show that there exists a constant $a_1 > 0$ such that, for a large enough $M$, the number $\mathcal{N} = \mathcal{N}(M)$ of $\gamma$-separated, $M$-rapid disks satisfies $$\mathcal{N} < a_1 t,$$ 
%Since we have chosen $\xi < 1$, this implies $\displaystyle \mathcal{N} < 2 a_1 \beta$, which is the original statement.\\

%For simplicity, we set the scaling factor to $\gamma := \delta^{-\frac{2}{3}}$. 
Using the same hypotheses, we will actually prove a slightly different statement. We let  $t := \beta + 1$. It follows from the fact that $\delta \beta^* < \frac{1}{2}$ that
\begin{equation}\label{relation_delta_t}
\delta t < 1. 
\end{equation} 
We normalize $F$ by the condition $\sup_{3\mathbb{D}} |F| = 1$, which has no effect whatsoever on the growth exponent. Finally, we can choose the uniform norm of the potential to be conveniently small : $||q||_\infty < \epsilon_0 < 1$. We will show that there exists a constant $c_5 > 0$ such that, for a large enough $M=M_0$, the number $\mathcal{N} = \mathcal{N}(M)$ of $\gamma$-separated, $M$-rapid disks satisfies $$\mathcal{N} < c_5 t,$$
which implies the result, since $t \leq 2 \beta^* = 2 \max\{\beta, 1\}.$
We recall that we are still in the setting of disks and annuli described in section \ref{s4ss1}, that is we have an arbitrary, finite collection of open disks $D_\nu \subset \frac{1}{60}\mathbb{D},\, 1 \leq \nu \leq N,$ each of radius $\delta$. Moreover, the collection of disks is $\gamma$-separated : the disks are mutually disjoint after a scaling of factor $\gamma$: $$|z_\mu - z_\nu| \geq 2 \gamma\delta, \text{for all } \mu \neq \nu,$$where $\gamma = \delta^{-\frac{1}{2}}$.

\subsection{A Carleman type estimate.}
The starting point of the proof is equation (2.4) of \cite{DF2}, which is an estimate in the spirit of Carleman, relating weighted $L^2$ norms of a function with that of some of its derivatives. 

\begin{lemma}\label{Carleman}
Let $t > 0$ and define 
\begin{equation*}
P(z) := \prod\limits_\nu (z-z_\nu).
\end{equation*} 
There exists a constant $c_{10} > 0$ such that, for any $f \in C_0^\infty\left( 3\mathbb{D} \setminus \bigcup\limits_\nu D_\nu (a) \right)$, we have
\begin{equation}\label{carl1}\tag{C1}
\int\limits_{3 \mathbb{D}} |\Delta f|^2  |P|^{-2} e^{t|z|^2} \geq c_{10} \left( t^2 \int\limits_{3\mathbb{D}} |f|^2 |P|^{-2} e^{t|z|^2}  + \delta^{-2}\int\limits_A |\nabla f|^2 |P|^{-2} e^{t|z|^2}\right). 
\end{equation}
\end{lemma}
The rather long development of that inequality is postponed to section \ref{s6}. Our first goal is to replace $|\nabla f|^2$ by $|f|^2$ in the right-hand side of the Carleman estimate. To do so, we will need the following two lemmas:
\begin{lemma}\label{prop1_lemma2}
There exist positive constants $c_i,\, i=11,...14$ such that, for any $w_1, w_2 \in A_\nu$, the following holds:
$$\text{(i)}\; c_{11} \leq \frac{e^{t|w_1|^2}}{e^{t|w_2|^2}} \leq c_{12}, \;\;\;\text{(ii)}\; c_{13} \leq \frac{|P(w_1)|}{|P(w_2)|} \leq c_{14}.$$
\end{lemma}

\begin{proof}

Since $\displaystyle w_1, w_2 \in \frac{1}{60}\mathbb{D}$, we have $$|t|w_1|^2 - t |w_2|^2| = t |(|w_1| - |w_2|)(|w_1 + w_2|)| \leq t | |w_1|-|w_2|| \leq t | w_1 - w_2| \leq 2 t \delta.$$ Since $t\delta \leq 1$, the result (i) now follows from exponentiation.

We now prove (ii). We have
\begin{align*}
|\log{|P(w_1)|} - \log{|P(w_2)|}| &= \left| \sum\limits_\mu \log{|w_1 - z_\mu}| - \sum\limits_\mu \log{|w_2 - z_\mu|}\right| \\
&\leq |\log{|w_1 - z_\nu|} - \log{|w_2 - z_\nu|}| \\
&+ \sum\limits_{\mu \neq \nu}| \log{|w_1 - z_\mu|} - \log{|w_2 - z_\mu|}|.
\end{align*} 
We first consider the first term of the right hand side of the above inequality. Suppose without loss of generality that $w_1$ is further from $z_\nu$ than $w_2$, that is $\displaystyle |w_1 - z_\nu| = \max{\{|w_1 - z_\nu|, |w_2 - z_\nu|}\}.$ Then, since both $w_1, w_2$ belong to the annulus $A_\nu$, we have
\begin{align*}
 |\log{|w_1 - z_\nu|} - \log{|w_2 - z_\nu|}| &= \log{|w_1 - z_\nu|} - \log{|w_2 - z_\nu|} \\
& \leq \log{(1 - a)\delta} - \log{(1-2a)\delta} \\
&= \log{\frac{1-a}{1-2a}} = a_2, 
\end{align*}
where $a_2 > 0$.
It now remains to estimate $\displaystyle \sum\limits_{\mu \neq \nu}| \log{|w_1 - z_\mu|} - \log{|w_2 - z_\mu|}|$. By the mean value theorem applied to $w \mapsto |w - z_\mu|,$ there exists some point $w \in \{ (1-\tau)w_1 + \tau w_2 : 0 \leq \tau \leq 1\}$ such that $$ |\log{|w_1 - z_\mu|} - \log{|w_2 - z_\mu|}| = |w-z_\mu|^{-1}|w_1 - w_2|.$$ 
The triangle inequality also implies $\displaystyle |z_\mu - z_\nu| \leq |w-z_\mu| + |w - z_\nu| \leq 2|w- z_\mu|,$ whence $\displaystyle |w - z_\mu|^{-1} \leq 2 |z_\mu - z_\nu|$ and $$|\log{|w_1 - z_\mu|} - \log{|w_2 - z_\mu|}| \leq 2 \frac{|w_1 - w_2|}{|z_\mu - z_\nu|} \leq \frac{4 \delta}{|z_\mu - z_\nu|}.$$ 
We now have
\begin{equation}\label{eq_lemma2_tmp1}
\sum\limits_{\mu \neq \nu} |\log{|w_1 - z_\mu|} - \log{|w_2 - z_\mu|}| \leq 4 \delta \sum\limits_{\mu \neq \nu} |z_\mu - z_\nu|^{-1}. 
\end{equation}
For $z \in \gamma D_\mu, \,\mu \neq \nu$, we have $\displaystyle |z-z_\nu| + |z_\mu - z_\nu| \leq 2 |z_\mu - z_\nu|$, from which we easily get 
\begin{equation*} 
\int_{\gamma D_\mu} |z - z_\nu|^{-1} \geq \frac{1}{2} \int_{\gamma D_\mu} \frac{1}{|z_\mu - z_\nu} = \frac{\pi (\gamma \delta)^2}{2 |z_\mu - z_\nu|}.
\end{equation*}
We define $\displaystyle E_\nu := \bigcup\limits_{\mu \neq \nu} \gamma D_\mu$ and we now have
\begin{equation}\label{eq_lemma2_tmp2}
4 \delta \sum_{\mu \neq \nu} |z_\mu - z_\nu|^{-1} \leq \frac{8 \delta}{\pi (\gamma \delta)^2} \sum_{\mu \neq \nu} \int_{\gamma D_\mu} |z - z_\nu|^{-1} = \frac{8}{\pi \gamma^2 \delta} \int_{E_\nu} |z - z_\nu|^{-1}.
\end{equation}
Let $B_\nu$ be the disk centred at $z_\nu$ whose total area is the same as $E_\nu$, that is $\text{Area}(B_\nu) = \text{Area}(E_\nu) = (N-1)\pi(\gamma \delta)^2.$ Remark that the maximum number of $\gamma$-separated disks of radius $\delta$ in $3\mathbb{D}$ is of the order $(\gamma \delta)^{-2}$; that is, there exists a positive constant $c$, independent of $\gamma$ and $\delta$, such that the cardinality $N$ of our collection of disks satisfies $N < c (\gamma \delta)^{-2}$. We consequently have
\begin{equation}\label{eq_lemma2_tmp3}
\int_{E_\nu} |z - z_\nu|^{-1} \leq \int_{B_{\nu}} |z - z_\nu|^{-1} \leq 4 \sqrt{\text{Area}(E_\nu)} \leq 4\sqrt{\pi N} \gamma \delta \leq 4 \sqrt{c \pi}.
\end{equation}
Combining equations (\ref{eq_lemma2_tmp1}), (\ref{eq_lemma2_tmp2}) and (\ref{eq_lemma2_tmp3}) now gives
\begin{equation*}
\sum\limits_{\mu \neq \nu} |\log{|w_1 - z_\mu|} - \log{|w_2 - z_\mu|}| \leq \frac{32 \sqrt{c \pi}}{\pi \gamma^2 \delta} = \frac{a_3}{\gamma^2 \delta} = a_3,
\end{equation*}
since $\gamma = \delta^{-\frac{1}{2}}$. Finally,
\begin{equation*}
|\log{|P(w_1)|} - \log{|P(w_2)|}| \leq a_2 + a_3,
\end{equation*}
from which the result follows via exponentiation. 

\end{proof}

The second lemma is a Poincar\'e like inequality:

\begin{lemma}\label{prop1_lemma3}
Suppose $f \in C^\infty(A_\nu)$ and vanishes on the inner boundary $|z| = (1-2a)\delta$ of $A_\nu$. Then,
\begin{equation}\label{Poincar}
\int\limits_{A_\nu} |\nabla f|^2 \geq \frac{c_{15}}{\delta^2} \int\limits_{A_\nu} |f|^2,
\end{equation}
where $c_{15}$ is a positive constant.
\end{lemma}

\begin{proof}
We introduce polar coordinates $(r,\theta)$ on $A_\nu$. Since $f((1-2a)\delta, \theta) \equiv 0$, the fundamental theorem of calculus yields
\begin{equation*}
f(r,\theta) = \int_{(1-2a)\delta}^r \frac{\partial f}{\partial s}(s,\theta)ds.
\end{equation*}
Hence,
\begin{equation*}
\int_{A_\nu} |f|^2 dA = \int_0^{2\pi} \int_{(1-2a)\delta}^{(1-a)\delta} \left( \int_{(1-2a)\delta}^r \frac{\partial f}{\partial s}(s,\theta) ds \right)^2 r\, dr d\theta.
\end{equation*}
By Cauchy-Schwarz, we have
\begin{equation*}
\left( \int_{(1-2a)\delta}^r \frac{\partial f}{\partial s}(s,\theta) ds \right)^2 \leq \int_{(1-2a)\delta}^r \left( \frac{\partial f}{\partial s}\right)^2 ds \int_{(1-2a)\delta}^r 1^2 ds \leq a\delta \int_{(1-2a)\delta}^r \left( \frac{\partial f}{\partial s}\right)^2 ds.
\end{equation*}
Consequently, 
\begin{align*}
\int_{A_\nu} |f|^2 &\leq a\delta \int_0^{2\pi} \int_{(1-2a)\delta}^{(1-a)\delta}  \int_{(1-2a)\delta}^{(1-a)\delta} \left( \frac{\partial f}{\partial s}\right)^2 ds\, r\, dr d\theta \\
&\leq a\delta \int_0^{2\pi} \left. \frac{r^2}{2} \right|_{(1-2a)\delta}^{(1-a)\delta} \int_{(1-2a)\delta}^{(1-a)\delta} \left( \frac{\partial f}{\partial s} \right)^2 \frac{s}{(1-2a)\delta} \,ds\, d\theta \\
&\leq c_{15} \delta^2 \int_0^{2\pi} \int_{(1-2a)\delta}^{(1-a)\delta} |\nabla f|^2 \,s\, ds\, d\theta = c_{15} \delta^2 \int_{A_\nu} |\nabla f|^2.
\end{align*}
\end{proof}

Fix one $w_\nu \in A_\nu$, for all $1 \leq \nu \leq N$. Then, for each $\nu$, we have

\begin{align*}
\int\limits_{A_\nu} |\nabla f|^2 |P|^{-2} e^{t|z|^2} &\geq (c_{11} c_{13}^2)  \frac{e^{t |w_\nu|^2}}{|P(w_\nu)|^2}\int\limits_{A_\nu} |\nabla f|^2 
\geq (c_{11} c_{13}^2 c_{15}) \frac{e^{t |w_\nu|^2}}{\delta^2 |P(w_\nu)|^2} \int\limits_{A_\nu} f^2 \\
&\geq (c_{11}^2 c_{13}^4 c_{15}) \delta^{-2} \int\limits_{A_\nu}  f^2 |P|^{-2} e^{t|z|^2}, 
\end{align*}
where we have used respectively Lemmas \ref{prop1_lemma2}, \ref{prop1_lemma3} and then \ref{prop1_lemma2} again. The Carleman estimate (\ref{carl1}) thus becomes

\begin{equation}\label{carl2}\tag{C2}
\int\limits_{3 \mathbb{D}} |\Delta f|^2  |P|^{-2} e^{t|z|^2} \geq a_4 \left( t^2 \int\limits_{3\mathbb{D}} f^2 |P|^{-2} e^{t|z|^2}  + \delta^{-4}\int\limits_A f^2 |P|^{-2} e^{t|z|^2}\right), 
\end{equation}
where $a_4 := \min\left\{c_{11}^2 c_{13}^4 c_{15}, c_{10}\right\}$.

\subsection{A suitable cut-off for F}

We now apply the previous estimate to $f=\theta F$, where $\theta$ is a suitable cut-off. More precisely, the cut-off $\theta$ satisfies the following properties:

\begin{itemize}
    \item [i.] $\displaystyle 0 \leq \theta \leq 1,\;\; \theta \in C_0^\infty\left( 2\mathbb{D} \setminus \cup_\nu D_\nu \right),$
    \item [ii.] $\displaystyle \theta(z) \equiv 1$ on $\left\{ z : |z| < 1,\; |z - z_\nu| > \left(1 - \frac{3}{2}a\right) \delta \right\}.$
    \item [iii.] $\displaystyle |\nabla \theta| + |\Delta \theta| \leq a_5$ on $\{|z| > 1 \}$.
    \item [iv.] $\displaystyle |\nabla \theta| \leq a_6 \delta^{-1},\;\; |\Delta \theta| \leq a_7 \delta^{-2}$ for $|z - z_\nu| \leq \left(1 - \frac{3}{2}a\right)\delta$.
\end{itemize}
\smallskip

The property (iv) allows us to control the growth properties of the cut-off in terms of the radius $\delta$ of the disks. Figure~\ref{fig_annuli_cutoff} summarizes the property of the cut-off.\\

%-------------------------------------------------
% DBUT FIGURE: ANNULI AND CUTOFF
%-------------------------------------------------
\begin{figure}[h]
\begin{center}
\includegraphics[scale=0.50]{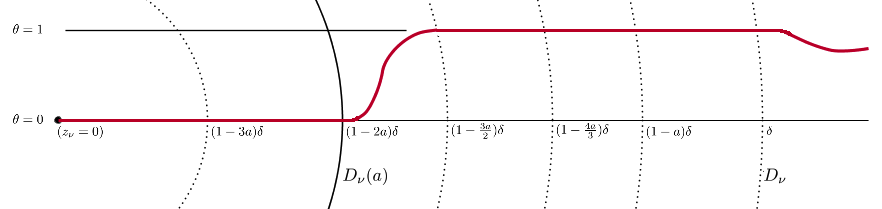}
\caption{A smooth cut-off $\theta$ defined on 2$\mathbb{D}$. }
\label{fig_annuli_cutoff}
\end{center}
\end{figure}
%-------------------------------------------------
% FIN FIGURE
%-------------------------------------------------

Using the properties of $\theta$, we have the following

\begin{lemma}\label{prop1_lemma4}
Let $F, \theta$ be as defined in our current setting. Then, 
\begin{equation*}
|\Delta\left(\theta F\right)| \leq 5 \left( q^2 F^2 + |\nabla \theta|^2 |\nabla F|^2 + F^2|\Delta \theta|^2 \right).
\end{equation*}
\end{lemma}

\begin{proof}
The proof is a simple computation:
\begin{align*}
|\Delta (\theta F) |^2 &= | \theta \Delta F + 2 \left( \nabla \theta \cdot \nabla F \right) + F \Delta \theta |^2 \\
&\leq \left(|\theta \Delta F| + 2 |\nabla \theta| |\nabla F| + |F \Delta \theta|  \right)^2 \\
&\leq 5\left( \theta^2|- q  F|^2 + |\nabla \theta|^2|\nabla F|^2 + F^2|\Delta \theta|^2\right) \\
&\leq 5\left(q^2 F^2 + |\nabla \theta|^2|\nabla F|^2 + F^2 |\Delta \theta|^2\right).
\end{align*}
\end{proof}
\medskip

Applying (\ref{carl2}) to $\theta F$ now yields 
\begin{equation*}
\int\limits_{2 \mathbb{D}} |\Delta \theta F|^2  |P|^{-2} e^{t|z|^2} \geq a_4 \left( t^2 \int\limits_{2\mathbb{D}} |\theta F|^2 |P|^{-2} e^{t|z|^2}  + \delta^{-4}\int\limits_A |\theta F|^2 |P|^{-2} e^{t|z|^2}\right).
\end{equation*}

Using Lemma (\ref{prop1_lemma4}) to estimate the (LHS) of the above equation, we now get
\begin{align*}
\int\limits_{2 \mathbb{D}} \left( q^2 F^2 + |\nabla \theta|^2 |\nabla F|^2 + F^2|\Delta \theta|^2\right)  |P|^{-2} e^{t|z|^2} &\geq  \\
\frac{a_4}{5} \left(
t^2 \int\limits_{2\mathbb{D}} |\theta F|^2 |P|^{-2} e^{t|z|^2}  
+ 
\delta^{-4}\int\limits_A |\theta F|^2 |P|^{-2} e^{t|z|^2}\right).
\end{align*}

Now, since our potential is small, $||q||_\infty < \epsilon_0$, the first term of the (LHS) can without loss of generality (by picking a smaller constant if needed) be absorbed by the (RHS), yielding
\begin{align*}\label{carl3}
\int\limits_{2 \mathbb{D}} \left( |\nabla \theta|^2 |\nabla F|^2 + F^2|\Delta \theta|^2\right)  |P|^{-2} e^{t|z|^2} &\geq  \\
a_8 \left(
t^2 \int\limits_{2\mathbb{D}} |\theta F|^2 |P|^{-2} e^{t|z|^2}  
+ 
\delta^{-4}\int\limits_A |\theta F|^2 |P|^{-2} e^{t|z|^2}\right).\tag{C3}
\end{align*}

The remainder of the proof consists mostly in improvements of the left and right hand sides of this last estimate.

\subsection{Using elliptic theory to improve the left hand side of (\ref{carl3}).}

We now work on the left-hand side of estimate the last Carleman estimate. By definition of the cut-off $\theta$, we have $|\nabla \theta| = | \Delta \theta| \equiv 0$ on $\displaystyle 2\mathbb{D} \setminus \left( A=\cup_\nu A_\nu \cup \left\{ 1 \leq |z| \leq 2\right\}\right),$ so that it makes sense to write (LHS) $ = I + \sum\limits_\nu I_\nu$, where $$ I = \int\limits_{1 < |z| < 2} \zeta(z),\;\;\; I_\nu = \int\limits_{A_\nu} \zeta(z),$$ and  $$\zeta(z) =  \left( |\nabla \theta|^2 |\nabla F|^2 + F^2|\Delta \theta|^2\right) |P|^{-2} e^{t|z|^2}.$$

The following lemma uses elliptic theory to improve estimates on both $I$ and $I_\nu$.
\begin{lemma}\label{prop1_lemma5}
There exist positive constants $c_{11}, c_{12}$ such that
\begin{itemize}
    \item [i.] $\displaystyle I \leq c_{16} e^{4t} \max\limits_{|z| \geq 1} |P|^{-2}\int\limits_{\frac{3}{4} < |z| < \frac{9}{4}} F^2$, \\
    \item [ii.] $\displaystyle I_\nu \leq c_{17} \delta^{-4} \max\limits_{A_\nu}  \left( |P|^{-2} e^{t|z|^2} \right) \int\limits_{A_\nu'} F^2.$
\end{itemize}
\end{lemma}

\begin{proof}
Recalling the various assumptions on the cutoff $\theta$ , we immediately have
\begin{align}\label{estimate_i1}
I &\leq a_5 e^{4t} \max_{1 \leq |z| \leq 2}{|P|^{-2}}\int_{1 < |z| < 2} \left( F^2 + |\nabla F|^2\right) \notag\\
&= a_5 e^{4t}  \max_{1 \leq |z| \leq 2}{|P|^{-2}}||F||^2_{H^1(\Omega')},
\end{align}
where $H^1 = W^{1,2}$ is the habitual Sobolev space and $\Omega' = \left\{Ê1 < |z| < 2 \right\}$. We now apply Theorem 8.8 in \cite{GiTr} with $L= \Delta$, $u=F$ and $f = -q F$ to get
%\begin{theorem}
%Let $\Omega \subset \mathbb{R}^2$ be an open connected domain and $u \in W^{1,2}(\Omega)$ be a weak solution of the equation $Lu = f$ in $\Omega$, where $L$ is a second %order strictly elliptic differential operator in $\Omega$, the coefficients $a^{ij}, b^i; i,j=1,2$ are uniformly Lipschitz continuous in $\Omega$, the coefficients $c^i, d; i=1,2$ are %essentially bounded in $\Omega$ and the function $f$ is in $L^2(\Omega)$. Then, for any subdomain $\Omega'$ strictly included in $\Omega$, we have $u \in W^{2,2}(\Omega')$ %and $$ ||u||_{W^{2,2}(\Omega')} \leq  c_{13} \left( ||u||_{L^2(\Omega)} + ||f||_{L^2(\Omega)}\right),$$
%where $c_{13}$ is a positive constant.
%\end{theorem}
%Note that by strictly included, we mean $$\displaystyle \sup\limits_{x \in \partial\Omega, y \in \Omega' } | x - y | > 0.$$ With $L= \Delta$, $u=F$ and $f = -p \xi F$, the theorem %immediately yields 
\begin{align*}
||F||_{W^{2,2}(\Omega')} &\leq a_9 \left( ||F||_{L^2(\Omega)} +  ||qF||_{L^2(\Omega)} \right) \\
&\leq  a_9 \max{\{1, \text{Area}(\Omega)\epsilon_0\}}||F||_{L^2(\Omega)}\\
&= a_{10} ||F||_{L^2(\Omega)},
\end{align*}
which holds for any subdomain $\Omega$ such that $\Omega' \subset \subset \Omega$, that is $$\displaystyle \sup\limits_{x \in \partial\Omega, y \in \Omega' } | x - y | > 0.$$ 

We set $\Omega := \{3/4 < |z| < 9/4 \}$ so that the above condition is satisfied. Since $|| \cdot ||_{W^{1,2}} \leq || \cdot ||_{W^{2,2}}$, we have $$||F||^2_{H^1(\Omega')} \leq a_{10}^2  ||F||^2_{L^2(\Omega)},$$ so that estimate (\ref{estimate_i1}) becomes 
\begin{equation*}
I \leq (a_5 a_{10}^2) e^{4t}  \max_{1 \leq |z| \leq 2}{|P|^{-2}}||F||^2_{L^2(\Omega)} = c_{11} e^{4t} \max_{|z| \geq 1} |P|^{-2} \int\limits_{\frac{3}{4} < |z| < \frac{9}{3}} F^2.
\end{equation*}

%----------------------
% LEMMA PART II
%-----------------------

We now prove the second part of the lemma. We define $ \bar{A}_\nu:=(1-2a)\delta < |z-z_\nu| < (1-3a/2)\delta \subset A_\nu$. Since $\theta(z) \equiv 1$ for $(1-3a/2)\delta < |z| < (1-a)\delta$, we have 
\begin{align}\label{lemma_elliptic_temp_eq1}
I_\nu &\leq \max_{A_\nu}{\left(|P|^{-2}e^{t|z|^2}\right)}  \int\limits_{\bar{A}_\nu} \left( |\nabla \theta|^2 |\nabla F|^2 + F^2 |\Delta \theta|^2\right) \notag\\ 
&\leq \max\left\{a_6^2, a_7^2\right\}\max_{A_\nu}{\left(|P|^{-2}e^{t|z|^2}\right)} \left[ \;\int\limits_{\bar{A}_\nu} \delta^{-2}|\nabla F|^2 + \int\limits_{\bar{A}_\nu}\delta^{-4}F^2 \right] .
\end{align}

Our goal is now to get rid of the gradient in the first integral of the last equation above. To do so, we set $\displaystyle \bar{I}_\nu := \int_{\bar{A}_\nu} |\nabla F|^2$ and introduce another cutoff $\phi \in C_0^\infty(A_\nu')$ which satisfies:
\begin{itemize}
    \item [i.] $\displaystyle 0 \leq \phi \leq 1$ 
    \item [ii.] $\displaystyle \phi(z) \equiv 1$ on $\bar{A}_\nu$ 
    \item [iii.] $\displaystyle |\nabla \phi| \leq a_{11}(\phi\delta^{-1})$. 
\end{itemize}
%---------------------------------------------------------------------------
% DBUT FIGURE: ANNULI AND SECOND CUTOFF
%----------------------------------------------------------------------------
\begin{figure}[h]
\begin{center}
\includegraphics[scale=0.35]{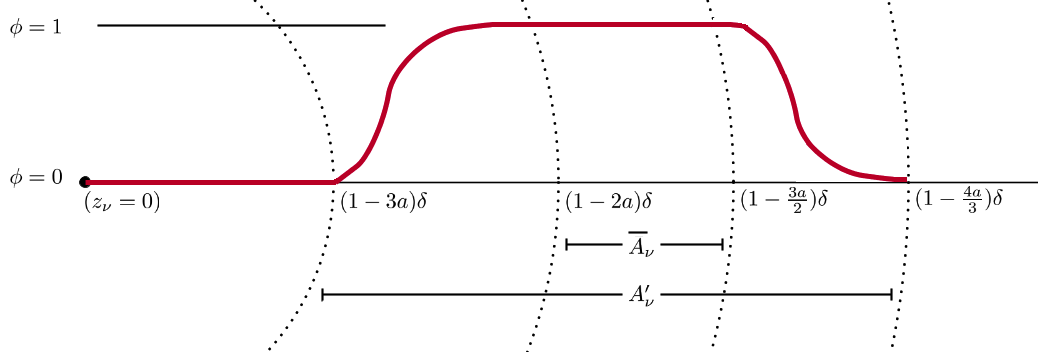}
\caption{A second cutoff $\phi$ on the annuli. }
\label{fig_cutoff2}
\end{center}
\end{figure}
%-------------------------------------------------
% FIN FIGURE
%-------------------------------------------------

Using Green's identity and since $\phi$ vanishes on the boundary of $A'_\nu$, we notice that $$ \int_{A_\nu'} \, q \, \phi F^2 = - \int_{A_\nu'} \phi F \Delta F = \int_{A_\nu'} \nabla (\phi F) \cdot \nabla F = \int_{A_\nu'}  F(\nabla F \cdot \nabla \phi) + \int_{A_\nu'} \phi |\nabla F|^2.$$

Thus, since $||q||_{\infty} < 1$, we get
\begin{equation}\label{lemma_elliptic_temp_eq2}
\int_{A_\nu'} \phi |\nabla F^2| \leq \int_{A_\nu'} \phi F^2 + \frac{a_{11}}{ \delta} \int_{A_\nu'}  \phi |F| \, ||\nabla F|| .
\end{equation}

Now, remark that for any non-negative numbers $a,b,c$ and $k > 0$, we have the following elementary inequality $abc \leq \frac{1}{2}\left(\frac{ab^2}{k} - kac^2\right),$ which we apply to our setting to get $$\phi \left( \frac{|F|}{\delta}\right) ||\nabla F|| \leq \frac{1}{2}\left(\frac{\phi F^2}{k\delta^2} + k\phi |\nabla F|^2\right).$$ We integrate over $A'_\nu$ and then choose $k$ small enough to absorb $\displaystyle \frac{1}{2}\left( k\phi |\nabla F|^2\right)$ in the left-hand side of equation ($\ref{lemma_elliptic_temp_eq2}$), so that it becomes 
\begin{equation*}
\int_{A_\nu'} \phi | \nabla F |^2 \leq \max\left\{1, \frac{a_{11}}{2k}\right\}  \frac{1}{\delta^2} \int_{A_\nu'} \phi F^2.
\end{equation*}
Going back to the definition of $\bar{I}_\nu$, we now have
\begin{equation*}
\bar{I}_\nu = \int_{\bar{A}_\nu} |\nabla F|^2 \leq \int_{A_\nu'} \phi |\nabla F|^2  \leq a_{12}\frac{1}{\delta^2} \int_{A_\nu'} \phi F^2 \leq a_{12}\, \delta^{-2} \int_{A_\nu'} F^2.
\end{equation*}

Plugging this into ($\ref{lemma_elliptic_temp_eq1}$) yields
\begin{align*} 
I_\nu &\leq \max\left\{a_6^2, a_7^2\right\}\max_{A_\nu}{\left(|P|^{-2}e^{t|z|^2}\right)} \left[ \delta^{-2} \bar{I}_\nu + \int\limits_{\bar{A}_\nu}\delta^{-4}F^2 \right]  \\
&\leq   \left( \max\left\{a_6^2, a_7^2\right\} \max\left\{1, a_{12}\right\} \right) \max_{A_\nu}{\left(|P|^{-2}e^{t|z|^2}\right)} \delta^{-4} \int\limits_{\bar{A}_\nu}F^2 \\ 
&\leq c_{17}  \max_{A_\nu}{\left(|P|^{-2}e^{t|z|^2}\right)} \delta^{-4} \int\limits_{A'_\nu}F^2.
\end{align*} 

\end{proof}

By Lemma \ref{prop1_lemma2}, we have $$\displaystyle \max_{A_\nu} \left( |P|^{-2} e^{t|z|^2}\right) \leq a_{13} \min_{A_\nu} \left( |P|^{-2} e^{t|z|^2}\right).$$ Applying the estimates of  Lemma \ref{prop1_lemma5} to the left-hand side of ($\ref{carl3}$) then gives
\begin{equation}\label{eq_temp_carl}
 (LHS) = I + \sum\limits_\nu I_\nu \leq a_{14}\left(e^{4t} \max\limits_{|z| \geq 1} |P|^{-2}\int\limits_{\frac{3}{4} < |z| < \frac{9}{4}} F^2 + \delta^{-4}\sum\limits_\nu \min\limits_{A_\nu} \left( |P|^{-2} e^{t|z|^2} \right) \int\limits_{A_\nu'} F^2\right),
\end{equation}
where $a_{14} = \max\left\{c_{16}, c_{16} a_{13}\right\}$.\\

The next lemma introduces the growth exponent $\beta$ of F in an expression which links the $L^2$ norms of $F$ on two annuli of different sizes.
\begin{lemma}\label{prop1_lemma6}
There exist a positive constant $c_{18}$ such that $$\int\limits_{3/4 < |z| < 9/4} F^2 \leq c_{18} e^{2 \beta} \int\limits_{1/4 < |z| < 1/2} F^2.$$
\end{lemma}

\begin{proof}
%Before we start, let us emphasize that, for the duration of the present proof, we will derogate slightly from our usual notation convention and write $\iint$ for double integrals and $\int$ for line integrals. 
First, recall that the potential $q$ satisfies $||q||_{\infty} < \epsilon_0$. On the one hand, we have: 
\begin{equation}\label{lemma_prop1_doub_exp_eq1}
\iint\limits_{3/4 < |z| < 9/4} F^2 dA < \iint\limits_{\frac{5}{2} \mathbb{D}} F^2 dA \leq \left( \frac{25\pi}{4} \right) \sup\limits_{\frac{5}{2} \mathbb{D}} \,F^2.
\end{equation}
On the other hand, the definition of the growth exponent yields 
\begin{equation}\label{lemma_prop1_doub_exp_eq2}
\sup\limits_{\frac{5}{2} \mathbb{D}} \,F^2 = e^{2\beta} \sup\limits_{\frac{1}{4}\mathbb{D}} \, F^2.
\end{equation}
Following a similar approach as Lemma 4.9 in \cite{NPS}, we now represent $F$ as the sum of its Green potential and Poisson integral. More precisely, for $|z| < {1/4}$ and given any fixed radius $\rho \in (\frac{1}{4},\frac{1}{2}]$, we have
\begin{equation}\label{lemma_prop1_doub_exp_eq3}
F(z) = \iint\limits_{\rho\mathbb{D}} p(\zeta) F(\zeta) G_\rho(z,\zeta) dA(\zeta) + \int\limits_{\rho \,\mathbb{S}^1} F(\zeta)  P_\rho(z, \zeta) ds(\zeta),
\end{equation}
where $\displaystyle G_\rho(z, \zeta) = \log{\left| \frac{\rho^2 - z\bar{\zeta}}{\rho(z - \zeta)} \right|}$ and $\displaystyle P_\rho(z, \zeta) = \frac{\rho^2 -|z|^2}{|\zeta - z|^2}$.
We respectively write $I_1$ and $I_2$ for the double integral and the (line) integral above and notice that
\begin{equation}\label{lemma_prop1_doub_exp_eq4}
F^2= I_1^2 + 2 I_1 I_2 + I_2^2 \leq 4(I_1^2 + I_2^2).
\end{equation}
Using Cauchy-Schwartz, we get the following upper bound:
\begin{align}
I_1^2 &\leq  \iint\limits_{\rho \mathbb{D}}  p^2(\zeta) F^2(\zeta) dA(\zeta)  \iint\limits_{\rho \mathbb{D}} G_\rho^2(z,\zeta) dA(\zeta)\notag\\
&\leq a_{15} \iint\limits_{\rho \mathbb{D}}  p^2(\zeta) F^2(\zeta) dA(\zeta) \leq a_{15} ||p||_{\infty}^2  \iint\limits_{\rho \mathbb{D}} F^2(\zeta) dA(\zeta) \notag\\
&\leq a_{15} \, \epsilon_0^2  \iint\limits_{\frac{1}{2} \mathbb{D}} F^2(\zeta) dA(\zeta). \label{lemma_prop1_doub_exp_eq5}
\end{align}
In the above, we have $\displaystyle a_{15} = \sup_{\rho \in (\frac{1}{4}, \frac{1}{2}]} \sup_{z \in \frac{1}{4}\mathbb{D}}  \iint\limits_{\rho \mathbb{D}} G_\rho^2(z, \zeta) dA(\zeta).$ Similarly,
\begin{equation}\label{lemma_prop1_doub_exp_eq6}
I_2^2 \leq \int\limits_{\rho \,\mathbb{S}^1} F^2(\zeta) ds(\zeta) \int\limits_{\rho \,\mathbb{S}^1} P_\rho^2(z,\zeta) ds(\zeta) \leq a_{16} \int\limits_{\rho \,\mathbb{S}^1} F^2(\zeta) ds(\zeta),
\end{equation}
with $\displaystyle a_{16} =  \sup_{\rho \in (\frac{1}{4}, \frac{1}{2}]} \sup_{z \in \frac{1}{4}\mathbb{D}}  \int\limits_{\rho \,\mathbb{S}^1} P_\rho^2(z,\zeta) ds(\zeta).$ 

\noindent Now, recalling that the representation of $F$ in ($\ref{lemma_prop1_doub_exp_eq3}$) holds for any $|z| \leq \frac{1}{4}$ and substituting ($\ref{lemma_prop1_doub_exp_eq5}$), ($\ref{lemma_prop1_doub_exp_eq6}$) in ($\ref{lemma_prop1_doub_exp_eq4}$), we get:
\begin{equation*}
\sup_{z \in \frac{1}{4}\mathbb{D}} F^2 \leq a_{17} \left( \epsilon_0^2 \iint\limits_{\frac{1}{2}\mathbb{D}} F^2 dA + \int\limits_{\rho \,\mathbb{S}^1} F^2 ds\right), \; \forall \rho \in \left(\frac{1}{4},\frac{1}{2}\right],
\end{equation*}
with $a_{17} = 4 \max\{a_{15}, a_{16}\}.$ Averaging over all $\rho$ yields:
\begin{align}
\sup_{z \in \frac{1}{4}\mathbb{D}} F^2 &\leq a_{17}  \frac{16}{3\pi} \left( \epsilon_0^2 \iint\limits_{\frac{1}{2}\mathbb{D}} F^2 dA +\iint\limits_{\frac{1}{4} < |z| < \frac{1}{2}} F^2 dA\right) \notag\\
&= a_{18}  \epsilon_0^2 \iint\limits_{\frac{1}{4}\mathbb{D}} F^2 dA + a_{18}(1 + \epsilon_0^2) \iint\limits_{\frac{1}{4} < |z| < \frac{1}{2}} F^2 dA \notag\\
&\leq \left( \frac{a_{18} \pi}{16}\right) \epsilon_0^2 \sup_{z \in \frac{1}{4}\mathbb{D}} F^2 + a_{18}(1 + \epsilon_0^2) \iint\limits_{\frac{1}{4} < |z| < \frac{1}{2}} F^2 dA \notag\\
&= a_{19} \epsilon_0^2 \sup_{z \in \frac{1}{4}\mathbb{D}} F^2 + a_{18}(1 + \epsilon_0^2) \iint\limits_{\frac{1}{4} < |z| < \frac{1}{2}} F^2 dA \label{lemma_prop1_doub_exp_eq7}
\end{align}
Hence,
\begin{equation*}
(1 - a_{19} \epsilon_0^2) \sup_{z \in \frac{1}{4}\mathbb{D}} F^2 \leq a_{18}(1 + \epsilon_0^2) \iint\limits_{\frac{1}{4} < |z| < \frac{1}{2}} F^2 dA.
\end{equation*}
It suffices to choose $\epsilon_0$ small enough so that $(1 - a_{19} \epsilon_0^2)$ is positive to finally obtain
\begin{equation}\label{lemma_prop1_doub_exp_eq8}
\sup_{z \in \frac{1}{4}\mathbb{D}} F^2 \leq \frac{a_{18}(1 + \epsilon_0^2)}{1 - a_{19} \epsilon_0^2} \iint\limits_{\frac{1}{4} < |z| < \frac{1}{2}} F^2 dA.
\end{equation}
Linking ($\ref{lemma_prop1_doub_exp_eq1}$), ($\ref{lemma_prop1_doub_exp_eq2}$) and ($\ref{lemma_prop1_doub_exp_eq8}$) together concludes the proof.
\end{proof}

To finalize our estimate of the left-hand side of ($\ref{carl3}$), we need a last lemma.
\begin{lemma}\label{prop1_lemma7}
Let $N$ be the number of disks $D_\nu$ in our collection, that is $N = \text{deg } P$. Then, there exists a positive constant $c_{19}$ such that
$$\max\limits_{z \geq 1} |P|^{-2} \leq e^{-c_{19} N} \min\limits_{|z| \leq \frac{1}{2}} |P|^{-2}.$$
 \end{lemma}

\begin{proof}
For $|z| \geq 1$, we have $$\frac{1}{|z - z_\nu|} \leq \frac{1}{|z| - |z_\nu|} \leq \frac{1}{1 - 1/60} = \frac{60}{59},$$ while for$|z| \leq 1/2$, we have $$\frac{1}{|z-z_\nu|} \geq \frac{1}{|z| + |z_\nu|} \geq \frac{1}{1/2 + 1/60} = \frac{60}{31}.$$ As a consequence, $$\max_{|z| \geq 1} |P|^{-2} \leq \left( \frac{60}{59} \right)^{2 \text{deg} P} = \left( \frac{31}{59} \right)^{2 \text{deg} P}\left( \frac{60}{31} \right)^{2 \text{deg} P} \leq   \left( \frac{31}{59} \right)^{2 \text{deg} P} \min_{|z| \leq 1/2} |P|^{-2}.$$ We set $c_{19} = - 2 \log \left(\frac{31}{59}\right)$ to conclude the proof. 
\end{proof}

Applying the results of the last two lemmas to equation ($\ref{eq_temp_carl}$), we obtain a final estimate for the (LHS) of (\ref{carl3}):
\begin{align}\label{lhs_final}
(LHS) &\leq a_{14}\left(e^{4t} e^{-c_{19}N}\min\limits_{|z| \leq \frac{1}{2}} |P|^{-2} c_{18}e^{2\beta} \int\limits_{\frac{1}{4} < |z| < \frac{1}{2}} F^2 + \delta^{-4}\sum\limits_\nu \min\limits_{A_\nu} \left( |P|^{-2} e^{t|z|^2} \right) \int\limits_{A_\nu'} F^2\right)\notag\\
&\leq a_{20}\left(e^{6t - c_{19}N} \min\limits_{|z| \leq \frac{1}{2}} |P|^{-2}  \int\limits_{\frac{1}{4} < |z| < \frac{1}{2}} F^2 + \delta^{-4}\sum\limits_\nu \min\limits_{A_\nu} \left( |P|^{-2} e^{t|z|^2} \right) \int\limits_{A_\nu'} F^2\right),
\end{align}
since $\beta < t$ and where $a_{20} = a_{14}\max\left\{c_{18},1\right\}$.
% and $a_{21} = \max\left\{4, c_{20}\right\}$.

\subsection{Improving the right-hand side of (\ref{carl3})}

Recalling that $t > 1$ as well as the various properties of the cut-off, we estimate the (RHS) of (\ref{carl3}):
\begin{align}\label{rhs_final}
a_{8}^{-1}\text{(RHS)} &= 
t^2 \int\limits_{2\mathbb{D}} |\theta F|^2 |P|^{-2} e^{t|z|^2}  + \delta^{-4}\sum\limits_\nu \int\limits_{A_\nu} |\theta F|^2 |P|^{-2} e^{t|z|^2}\notag\\
&\geq \left(\frac{\pi}{4}\right) \min\limits_{|z| \leq \frac{1}{2}} |P|^{-2} \int\limits_{|z| < \frac{1}{2}} F^2 + \delta^{-4}\sum\limits_\nu \min\limits_{A_\nu} \left( |P|^{-2}e^{t|z|^2}\right) \int\limits_{A_\nu}F^2\notag\\
&\geq a_{21} \left( \min\limits_{|z| < \frac{1}{2}} |P|^{-2} \int\limits_{\frac{1}{4} < |z| < \frac{1}{2}} F^2 + \delta^{-4}\sum\limits_\nu \min\limits_{A_\nu} \left( |P|^{-2}e^{t|z|^2}\right) \int\limits_{A^{''}_\nu}F^2\right),
\end{align}
where $a_{21} = \min\left\{\frac{\pi}{4},1\right\}$.

\subsection{Conclusion}

At last, putting together the estimates (\ref{carl3}), (\ref{lhs_final}), (\ref{rhs_final}) yields
\begin{align*}
e^{6t - c_{19}N} \min\limits_{|z| \leq \frac{1}{2}} |P|^{-2}  \int\limits_{\frac{1}{4} < |z| < \frac{1}{2}} F^2 + \delta^{-4}\sum\limits_\nu \min\limits_{A_\nu} \left( |P|^{-2} e^{t|z|^2} \right) \int\limits_{A_\nu'} F^2 &\geq \notag \\
a_{22} \left( \min\limits_{|z| < \frac{1}{2}} |P|^{-2} \int\limits_{\frac{1}{4} < |z| < \frac{1}{2}} F^2 + \delta^{-4}\sum\limits_\nu \min\limits_{A_\nu} \left( |P|^{-2}e^{t|z|^2}\right) \int\limits_{A^{''}_\nu}F^2 \right),
\end{align*}
where $\displaystyle a_{22} =  \frac{a_{21}a_{8}}{a_{20}}.$ Recall that a disk $D_\nu$ is said to be $M$-rapid if $$ M\int_{A'_\nu}F^2 \leq \int_{A^{''}_\nu} F^2.$$ Suppose now that all the disks of our collection are $M$-rapid, i.e. that $\mathcal{N} = N$ and assume without loss of generality that $a_{22} > 1$ (otherwise, the argument still works: it suffices to pick a larger $M$). We get

\begin{align}\label{final_estimate}
e^{6t - c_{19}N} \min\limits_{|z| \leq \frac{1}{2}} |P|^{-2}  \int\limits_{\frac{1}{4} < |z| < \frac{1}{2}} F^2 + \delta^{-4}\sum\limits_\nu \min\limits_{A_\nu} \left( |P|^{-2} e^{t|z|^2} \right) \int\limits_{A_\nu'} F^2 &\geq \notag \\
 \min\limits_{|z| < \frac{1}{2}} |P|^{-2} \int\limits_{\frac{1}{4} < |z| < \frac{1}{2}} F^2 + M \delta^{-4}\sum\limits_\nu \min\limits_{A_\nu} \left( |P|^{-2}e^{t|z|^2}\right) \int\limits_{A'_\nu}F^2 ,
\end{align}

We get a contradiction if $\displaystyle N >\frac{6}{c_{19}}t  \Longleftrightarrow c_{19} N > c_{6}t$ and the proof is completed.

%-----------------------------------------------------------------------------------------------------------------
% 
%		SECTION 6 : CARLEMAN INEQUALITY
%
%-----------------------------------------------------------------------------------------------------------------
\section{An inequality in the spirit of Carleman}\label{s6}
Carleman estimates are known to be useful in obtaining unique continuation results as well as growth estimates (see for instance \cite{KT}). It is thus not surprising that the estimate (\ref{carl1}) has played a crucial role in the proof of the growth estimate presented in the previous section. For completeness, we present here one way to obtain such an inequality, which follows very closely the approach taken by Donnelly and Fefferman in Section 2 of \cite{DF2}.

\subsection{An elementary inequality in a weighted Hilbert space}
We let $\mathcal{D} \subset \mathbb{C}$ be open, bounded and $\displaystyle \varphi: \mathcal{D} \rightarrow \mathbb{R}$ be a smooth real-valued function.  Let also $\mathcal{H} = L^2\left(\mathcal{D}, e^{-\varphi}\text{dxdy}\right)$ be the Hilbert space of complex valued square integrable functions on $\mathcal{D}$ with respect to the weight $e^{-\varphi}$. Finally, let $\displaystyle u \in C_0^\infty(\mathcal{D}) \subset H$. We introduce the following differential operators
\begin{equation*}
\partial := \frac{1}{2}\left(\frac{\partial}{\partial x} - i \frac{\partial}{\partial y}\right),\;\; \bar{\partial} := \frac{1}{2}\left(\frac{\partial}{\partial x} + i \frac{\partial}{\partial y}\right),\;\; \bar{\partial}^* := e^{-\varphi}\partial (e^{-\varphi} \cdot )\,.
\end{equation*}

Easy computations allow one to verify the following facts: 
\begin{itemize}
\item [i.] For any $\mathbb{R}$-valued function $\psi$, $\bar{\partial}\partial \psi = \frac{1}{4}\Delta\psi$.
\item[ii.]  By the Cauchy-Riemann equations, $u$ is holomorphic if and only if $\bar{\partial} u = 0$.
\item [iii.] $\bar{\partial}^*$ is the adjoint operator of $\bar{\partial}$.
\item [iv.] $[\bar{\partial}, \bar{\partial}^*]u = \left(\frac{1}{4}\Delta \varphi\right) u$, where the interior of the parenthesis acts on $u$ by multiplication.
\end{itemize}

\begin{lemma}\label{carleman_proof_ineq1}
Let $\Phi: \mathcal{D} \rightarrow \mathbb{R}$ be a smooth, positive function. Then, 
\begin{equation*}
\int\limits_{\mathcal{D}} |\bar{\partial}u|^2\Phi \geq \int\limits_{\mathcal{D}} \frac{1}{4}\left(\Delta \log{\Phi} \right) |u|^2 \Phi,
\end{equation*}
where the integrals are taken with respect to the usual Lebesgue measure, that is, not in the weighted Hilbert space $\mathcal{H}$.
\end{lemma}

\begin{proof}
Put $\varphi := - \log{\Phi}$, i.e. $e^{-\varphi} = \Phi$. In the following, the norms and inner products are taken in the Hilbert space $\mathcal{H}$:
\begin{align*}
0 \leq ||\bar{\partial}^* u ||^2 = \left(\bar{\partial}^*u,\, \bar{\partial}^*u\right) &= \left( \bar{\partial}\bar{\partial}^*u,\, u\right)   \\
&= \left(\bar{\partial}^*\bar{\partial}u,\,u\right) + \left([\bar{\partial},\, \bar{\partial}^*]u,\, u \right) \\
&= \left(\bar{\partial}u ,\, \bar{\partial} u\right) + \left([\bar{\partial},\, \bar{\partial}^*]u ,\, u \right) \\
&= || \bar{\partial} u ||^2 + \int\limits_\mathcal{D} \left(\frac{1}{4}\Delta \varphi \right) |u|^2 e^{-\varphi}.
\end{align*}
Thus, $\displaystyle ||\bar{\partial} u ||^2 \geq -\frac{1}{4}\int\limits_\mathcal{D} \left(\Delta \varphi\right) |u|^2 e^{-\varphi} = \frac{1}{4}\int\limits_\mathcal{D}\left( \Delta \log{\Phi}\right) |u|^2 \Phi.$
\end{proof}

\subsection{A specialized choice of weight function}
The remainder of the section aims to specialize the choice of $\Phi$ in order to obtain a more refined inequality. In particular, we will build a weight function which has singularities on a crucial set of points. In the following, $a$ is a small, positive constant: $0 < a \llÊ1$.
 \begin{lemma}\label{lemma_cutoff_carleman}
 There exists a function $\Psi_0(z)$, defined for $|z| > (1-2a)$, such that
 \begin{itemize}
\item [i.] $a_1 \leq \Psi_0(z) \leq a_2$, where $a_1,\, a_2 > 0$,
\item [ii.] $\Psi_0 (z) \equiv 1$ on $\{ |z| > 1\}$, 
\item [iii.] $\Delta \log \Psi_0 \geq 0$ on $\{|z| > (1-2a)\}$, 
\item [iv.] $\Delta \log \Psi_0 \geq a_3 > 0$ on $\{ 1-2a < |z| < 1-a \}$.
 \end{itemize}
 \end{lemma}

\begin{proof}
First, choose $\psi_0(z)$ to be a radial function, i.e. depending only on $r=|z|$. Let $h(r) \geq 0$ be smooth and such that $\displaystyle h(r) \geq a_3$ for $1-2a < r < 1-a$ and $\displaystyle h(r) = 0$ for $|z| > 1 - \frac{a}{2}$. Now consider the radial Laplacian $$\Delta \log{\psi_0}(r) = \left( \frac{d^2}{dr^2} + \frac{1}{r}\frac{d}{dr}\right)\log{\psi_0}(r),$$ which has smooth coefficients on $r > 1-2a$. By the fundamental theorem for ordinary differential equations, we let $\log{\psi_0}(r)$ be the solution of the second order ODE

\begin{eqnarray*}
  \left\{
  \begin{aligned}
	\Delta \log{\psi_0}(r) = h(r), \\
	\log{\psi_0}(1) = 0,\\
	\log{\psi_0}'(1)=0.
  \end{aligned}
  \right.
\end{eqnarray*}
The function $\psi_0$ satisfies all the requirements.
\end{proof}

We now let $\displaystyle D_\nu := \{z : |z - z_\nu| < \delta \},\; 1 \leq \nu \leq N$, denote a finite collection of disks in the open unit disk $\mathbb{D}$ and let $D_\nu(a)$ be the closure of $(1 - 2a)D_\nu$.  Define $\displaystyle \Phi_0: \mathbb{C} \setminus \cup_\nu D_\nu (a)$ by $$\Phi_0(z) = \twopartdef { 1 } {z \not\in \cup_\nu D_\nu} {\psi_0\left( \frac{z - z_\nu}{\delta}\right) } {z \in D_\nu.}$$
We have that $\log{\Phi_0(z)} = \log{\Psi_0(w(z))}$, where $\displaystyle w(z) = \frac{z-z_\nu}{\delta}$ and $\displaystyle w'(z) = \frac{1}{\delta}.$ Thus, $$ \Delta \log{\Phi_0(z)} = \frac{1}{\delta^2}\Delta \log{\psi_0(w(z))} \geq a_3,$$ for $z \in A_\nu = \{(1-2a)\delta < |z| < (1-a)\delta \}.$ By Lemma \ref{lemma_cutoff_carleman}, we have
\begin{itemize}
\item [i.] $\displaystyle a_1 \leq \Phi_0(z) \leq a_2$,
\item [ii.] $\displaystyle \Delta \log{\Phi_0} \geq 0,\;\forall z \in \mathbb{C} \setminus \cup_\nu D_\nu(a)$,
\item [iii.] $\displaystyle \Delta \log{\Phi_0} \geq \frac{a_3}{\delta^2},\;\forall z \in A_\nu(\delta)$.
\end{itemize}
\smallskip
Let $t > 0$ be a constant and denote by $A$ the union $\cup_\nu A_\nu (\delta)$. We want to apply Lemma \ref{carleman_proof_ineq1} to $\Phi(z) := \Phi_0(z)e^{t |z|^2}$. For $u \in C_0^\infty \left(\mathbb{C} \setminus \cup_\nu D_\nu(a) \right)$, we assume that $\mathcal{D}$ is a bounded domain such that $\text{supp } u \subset \mathcal{D}$ and $\displaystyle A \subset \mathcal{D} \subset \mathbb{C} \setminus \cup_\nu D_\nu(a).$ Applying the lemma gives

\begin{equation}
\int\limits_\mathcal{D} |\bar{\partial} u|^2 \Phi_0(z)e^{t|z|^2} \geq \int\limits_\mathcal{D} \frac{1}{4} \left( \Delta \log{\Phi_0 e^{t|z|^2}} \right) |u|^2 \Phi_0 e^{t|z|^2}. 
\end{equation}

But $\displaystyle \log{\Phi_0e^{t|z|^2}} = \log{\Phi_0} + t|z|^2$ and the right-hand side of the above inequality satisfies

\begin{align*}
\text{RHS} &= \left[ \int\limits_A + \int\limits_{\mathcal{D} \setminus A} \right] \frac{1}{4}\left(\Delta \log{\Phi_0}\right)|u|^2 \Phi_0(z)e^{t|z|^2} + t \int\limits_\mathcal{D} |u|^2\Phi_0(z)e^{t|z|^2} \\
&\geq \frac{a_3}{\delta^2}\int\limits_A |u|^2\Phi_0(z)e^{t|z|^2} + t \int\limits_\mathcal{D} |u|^2\Phi_0(z)e^{t|z|^2}. 
\end{align*}
Since $\Phi_0$ is bounded, we get
\begin{equation}\label{carleman_proof_intermediate_eq1}
\int\limits_\mathcal{D} |\bar\partial u|^2 e^{t|z|^2} \geq \frac{a_4}{\delta^2} \int\limits_A |u|^2e^{t|z|^2} + a_5 t \int\limits_\mathcal{D} |u|^2e^{t|z|^2}.
\end{equation}

Define the holomorphic function $\displaystyle P(z) := \prod\limits_\nu(z-z_\nu)$ and replace $ u \mapsto \frac{u}{P}$. Then, $$\bar\partial\left(\frac{u}{P} \right) = \frac{\bar\partial u P - u \bar\partial P}{P^2} = \frac{\bar\partial u}{P}.$$

Since $\frac{u}{P} \in C_0^\infty(\mathcal{D})$, equation (\ref{carleman_proof_intermediate_eq1}) becomes
\begin{equation}\label{carleman_proof_intermediate_eq2}
\int\limits_\mathcal{D} |\bar\partial u|^2 |P|^{-2} e^{t|z|^2} \geq \frac{a_4}{\delta^2} \int\limits_A |u|^2|P|^{-2}e^{t|z|^2} + a_5 t \int\limits_\mathcal{D} |u|^2|P|^{-2}e^{t|z|^2}.
\end{equation}

All of the above discussion is valid for $u: \mathcal{D} \rightarrow \mathbb{C}$. We now choose $f: \mathcal{D} \rightarrow \mathbb{R}$. We have $|\bar\partial f| = |\partial f| = |\nabla f|.$ We choose $u = \partial f$, whence $\displaystyle \bar\partial u = \bar\partial \partial f = \frac{1}{4}\Delta f$, which yields
\begin{equation*}
\int\limits_\mathcal{D} |\Delta f|^2 |P|^{-2} e^{t|z|^2} \geq \frac{a_6}{\delta^2} \int\limits_A |\nabla f|^2|P|^{-2}e^{t|z|^2} + a_7 t \int\limits_\mathcal{D} |\bar\partial f|^2|P|^{-2}e^{t|z|^2}.
\end{equation*}

We work on the last integral. Applying Lemma \ref{carleman_proof_ineq1} to $\Phi = |P|^{-2} e^{t|z|^2}$, we get $$ \int\limits_\mathcal{D} |\bar\partial f|^2 |P|^{-2} e^{t|z|^2} \geq \frac{1}{4}\int\limits_\mathcal{D} \left( \Delta \log{(|P|^{-2} e^{t|z|^2})} \right) |f|^2 |P|^{-2}e^{t|z|^2}.$$

Also, $\displaystyle \log{|P|^{-2}} = - \log\prod\limits_\nu |z-z_\nu|^2 = - \sum\limits_\nu \log{|z-z_\nu|^2}$, whence $$\Delta \log{(|P|^{-2}e^{t|z|^2})}= -\sum\limits_\nu \delta(z - z_\nu) + 4t,$$ where $\delta$ is the Dirac-delta, meaning that the sum above vanishes on $\mathcal{D}$. Thus, $$\int\limits_\mathcal{D} |\bar\partial f|^2 |P|^{-2} e^{t|z|^2} \geq t \int\limits_\mathcal{D} |f|^2|P|^{-2}e^{t|z|^2}.$$

Finally, equation (\ref{carleman_proof_intermediate_eq2}) becomes the desired Carleman estimate
\begin{equation}\tag{C1}
\int\limits_\mathcal{D} |\Delta f|^2 |P|^{-2} e^{t|z|^2} \geq \frac{a_6}{\delta^2} \int\limits_A |\nabla f|^2|P|^{-2}e^{t|z|^2} + a_7 t^2 \int\limits_\mathcal{D} |f|^2|P|^{-2}e^{t|z|^2},
\end{equation}
which holds for any $\displaystyle f \in C_0^\infty\left(\mathbb{R}^2 \setminus \cup_\nu D_\nu(a)\right)$, with $\mathcal{D}$ a bounded open set such that $A \subset \mathcal{D} \subset \cup_\nu D_\nu(a).$

%-----------------------------------------------------------------------------------------------------------------
% 
%		SECTION 7 : DISCUSSION
%
%-----------------------------------------------------------------------------------------------------------------
\section{Discussion.}\label{s7}

\subsection{Higher dimensions}
In this paper, we have studied eigenfunctions of the Laplace-Beltrami operator on closed $C^\infty$ surfaces and have underlined a natural interpretation of Yau's conjecture in light of Theorem \ref{thm1}. Since the conjecture is expected to hold in any dimension, it is natural to ask

\begin{question}
Does Theorem \ref{thm1} hold for a compact, smooth manifold of dimension $n \geq 3$ ?
\end{question}

It seems reasonable to expect that the result holds in higher dimension: on the one hand, as previously stated, Yau's conjecture on the size of nodal sets is formulated for manifolds of any dimensions. On the other hand, some fundamental results for the growth exponents of eigenfunctions are known to hold in any dimension, most notably the Donnelly-Fefferman growth bound 
\begin{equation}\label{DF_gb}
\beta(\phi_\lambda, B) = \log \frac{\sup_B |\phi_\lambda|}{\sup_{\frac{1}{2} B} |\phi_\lambda|} \leq c \sqrt{\lambda},
\end{equation}
where $B$ is any metric ball (see for instance \cite{DF1, M, NPS}). However, the approach we have used relies crucially on the reduction of an eigenfunction  $\phi_\lambda$ to a planar solution $F$ to a Schr\"{o}dinger equation, a transformation made possible by the existence of local conformal coordinates, a fact that does not generalize in dimensions $n \geq 3$. One would therefore need to follow a fundamentally different approach to prove a result in the spirit of Theorem \ref{thm1} in that setting. In \cite{NPS}, the authors give a simpler proof of the growth bound (\ref{DF_gb}) in the setting of closed surfaces. A generalization of that proof in higher dimensions has been done by Mangoubi in \cite{M}, using notably a clever extension of eigenfunctions on a $n$-dimensional manifold $M$ to harmonic functions on the $(n+1)$ dimensional manifold $M \times \mathbb{R}$ (see also \cite{Lin, JL, NPS}). We believe that a similar treatment could be useful in attempting to generalize Theorem $\ref{thm1}$.

\subsection{How to measure the growth: generalization to $L^q$ norms}

Our measure of the growth of eigenfunctions has been made through growth exponents defined on small metric disks on which we have taken the $L^\infty$ norm. Indeed, we recall that
\begin{equation*}
\beta_p(\lambda) = \log \frac{\sup_B |\phi_\lambda|}{\sup_{\alpha_0} B |\phi_\lambda|},
\end{equation*}
where $B$ is a metric ball of small radius centred at $p \in M$. For $1 \leq q \leq \infty$, define the more general $q$-growth-exponent $\beta^q_p(\lambda)$ of an eigenfunction $\phi_\lambda$ in the following way
\begin{equation*}
\beta_p^q(\lambda) := \log \frac{||\phi_\lambda||_{L^q(B)}} {||\phi_\lambda||_{L^q(\alpha_0B)} },
\end{equation*}
where $B$ is once again a suitably small metric ball centred at $p$. Notice that $\beta_p(\lambda) = \beta^\infty_p(\lambda).$ Consider the average of such quantities on the surface, that is, define
\begin{equation*}
B^q(\lambda) := \frac{1}{\text{Vol}(M)} \int_M \beta_p^q(\lambda) \mathrm{d} V_g,
\end{equation*}
and then ask 
\begin{question}
For which $q \in [1, \infty)$, if any, do we have the following analogue of Theorem $\ref{thm1}$:

\begin{equation*}
c  B^q(\lambda) \loh \leq \mathcal{H}^1(Z_\lambda) \leq C ( B^q(\lambda) \loh + 1).
\end{equation*}
\end{question}

Keeping our setting of closed surfaces, it would suffice to prove analogues of Theorems \ref{thm2}, \ref{thm3} for $q$-growth exponents of planar Schr\"{o}dinger eigenfunctions to answer positively the last question, but there does not seem to be an obvious way to tackle this problem.

%------------------------------------------------------------------------------------------------------------------------------------------------------------------------------------------
%
% BIBLIOGRAPHY
%
%------------------------------------------------------------------------------------------------------------------------------------------------------------------------------------------

\vspace{3 mm}

{\scshape Département de mathématiques et de statistique,
Université de Montréal, CP 6128 succ. Centre-Ville, Montréal,
H3C 3J7, Canada}

\emph{E-mail address:} \verb"groyfortin@dms.umontreal.ca"

\end{document}